\documentclass[11pt, a4paper]{amsart}
\usepackage[dvips]{graphicx}
\usepackage{amssymb, amstext, amscd, amsmath, amsthm, amsfonts, color}
\usepackage{epstopdf}
\usepackage{epsfig}
\usepackage{geometry} 
\usepackage{tikz,mathdots,enumerate,tabularx}
\usepackage{pgfplots}% This uses tikz
\usepackage{url}
\usepackage{caption}

\usepackage{kbordermatrix} % include package @ document preamble
\renewcommand{\kbldelim}{[} % change default array delimiters to parentheses
\renewcommand{\kbrdelim}{]}

\usepackage{blindtext}
\usepackage{url}
\usepackage{tikz}
\usetikzlibrary{shapes.geometric, calc}
\usepackage{array}
\usepackage{epsfig}
\usepackage{eucal}
\usepackage{latexsym}
\usepackage{mathrsfs}
\usepackage{textcomp}
\usepackage{verbatim}

\usepackage{hhline}

\usepackage{hyperref}

\begin{document}

%=================================================================
%% Please use the following mathematics environments:

\newtheorem{Theorem}{Theorem}[section]
\newtheorem{Lemma}[Theorem]{Lemma}
\newtheorem{Characterization}[Theorem]{Characterization}
\newtheorem{Proposition}[Theorem]{Proposition}
\newtheorem{Property}[Theorem]{Property}
\newtheorem{Problem}[Theorem]{Problem}
\newtheorem{Example}[Theorem]{Example}
\newtheorem{Remark}[Theorem]{Remark}
\newtheorem{Corollary}[Theorem]{Corollary}
\newtheorem{Definition}[Theorem]{Definition}
%% For proofs, please use the proof environment (the amsthm package is loaded by the MDPI class).

%
%      Blackboard bold letters
\newcommand{\bC}{{\mathbb{C}}}
\newcommand{\bD}{{\mathbb{D}}}
\newcommand{\bN}{{\mathbb{N}}}
\newcommand{\bQ}{{\mathbb{Q}}}
\newcommand{\bR}{{\mathbb{R}}}
\newcommand{\bT}{{\mathbb{T}}}
\newcommand{\bX}{{\mathbb{X}}}
\newcommand{\bZ}{{\mathbb{Z}}}
\newcommand{\bH}{{\mathbb{H}}}
%
%      Capital script letters
  \newcommand{\A}{{\mathcal{A}}}
  \newcommand{\B}{{\mathcal{B}}}
  \newcommand{\C}{{\mathcal{C}}}
  \newcommand{\D}{{\mathcal{D}}}
  \newcommand{\E}{{\mathcal{E}}}
  \newcommand{\F}{{\mathcal{F}}}
  \newcommand{\G}{{\mathcal{G}}}
\renewcommand{\H}{{\mathcal{H}}}
  \newcommand{\I}{{\mathcal{I}}}
  \newcommand{\J}{{\mathcal{J}}}
  \newcommand{\K}{{\mathcal{K}}}
\renewcommand{\L}{{\mathcal{L}}}
  \newcommand{\M}{{\mathcal{M}}}
  \newcommand{\N}{{\mathcal{N}}}
\renewcommand{\O}{{\mathcal{O}}}
\renewcommand{\P}{{\mathcal{P}}}
  \newcommand{\Q}{{\mathcal{Q}}}
  \newcommand{\R}{{\mathcal{R}}}
\renewcommand{\S}{{\mathcal{S}}}
  \newcommand{\T}{{\mathcal{T}}}
  \newcommand{\U}{{\mathcal{U}}}
  \newcommand{\V}{{\mathcal{V}}}
  \newcommand{\W}{{\mathcal{W}}}
  \newcommand{\X}{{\mathcal{X}}}
  \newcommand{\Y}{{\mathcal{Y}}}
  \newcommand{\Z}{{\mathcal{Z}}}
%
% Fraktur letters
\newcommand{\fA}{{\mathfrak{A}}}
\newcommand{\fB}{{\mathfrak{B}}}
\newcommand{\fC}{{\mathfrak{C}}}
\newcommand{\fD}{{\mathfrak{D}}}
\newcommand{\fE}{{\mathfrak{E}}}
\newcommand{\fF}{{\mathfrak{F}}}
\newcommand{\fG}{{\mathfrak{G}}}
\newcommand{\fH}{{\mathfrak{H}}}
\newcommand{\fI}{{\mathfrak{I}}}
\newcommand{\fJ}{{\mathfrak{J}}}
\newcommand{\fK}{{\mathfrak{K}}}
\newcommand{\fL}{{\mathfrak{L}}}
\newcommand{\fM}{{\mathfrak{M}}}
\newcommand{\fN}{{\mathfrak{N}}}
\newcommand{\fO}{{\mathfrak{O}}}
\newcommand{\fP}{{\mathfrak{P}}}
\newcommand{\fQ}{{\mathfrak{Q}}}
\newcommand{\fR}{{\mathfrak{R}}}
\newcommand{\fS}{{\mathfrak{S}}}
\newcommand{\fT}{{\mathfrak{T}}}
\newcommand{\fU}{{\mathfrak{U}}}
\newcommand{\fV}{{\mathfrak{V}}}
\newcommand{\fW}{{\mathfrak{W}}}
\newcommand{\fX}{{\mathfrak{X}}}
\newcommand{\fY}{{\mathfrak{Y}}}
\newcommand{\fZ}{{\mathfrak{Z}}}
\newcommand{\ul}{\underline  }
% Misc notation
\newcommand{\Aut}{\operatorname{Aut}}
\newcommand{\sgn}{\operatorname{sgn}}
\newcommand{\rank}{\operatorname{rank}}
\newcommand{\adj}{\operatorname{adj}}
\newcommand{\ran}{\operatorname{ran}}

%=================================================================

 \title[Almost periodic rigidity of crystallographic bar-joint frameworks]{The almost periodic rigidity of crystallographic bar-joint frameworks
}

\author[G. Badri, D. Kitson and S. C. Power]{G. Badri, D. Kitson and S. C. Power}
\thanks{DK and SCP supported by EPSRC grant  EP/J008648/1.}

%\emph{Crystal frameworks, almost periodic rigidity}

\address{Dept.\ Math.\ Stats.\\ Lancaster University\\
Lancaster LA1 4YF \\U.K. }

\email{g.badri@lancaster.ac.uk\\ d.kitson@lancaster.ac.uk \\  s.power@lancaster.ac.uk}

\begin{abstract} A crystallographic bar-joint framework $\C$ in $\bR^d$ is shown to be almost periodically infinitesimally rigid if and only if it 
is strictly periodically infinitesimally rigid and the rigid unit mode (RUM) spectrum $\Omega(\C)$ is a singleton.
Moreover the almost periodic infinitesimal flexes of $\C$ are characterised in terms of a matrix-valued function  $\Phi_\C(z)$ on the $d$-torus $\bT^d$  determined 
by a full rank translation symmetry group and an associated motif of joints and bars. 
\end{abstract}

\maketitle

\section{Introduction}

The rigidity of a crystallographic bar-joint framework $\C$ in the Euclidean spaces $\bR^d$ with respect to periodic first order flexes is determined by a finite matrix,  the associated periodic rigidity matrix. For essentially generic frameworks of this type in two dimensions there is a deeper  combinatorial characterisation which is  a counterpart of Laman's characterisation of the infinitesimal rigidity of generic placements of finite graphs in the plane.  See Ross \cite{ros-kavli}. For related results and characterisations of other forms of periodic infinitesimal rigidity see Borcea and Streinu \cite{bor-str, bor-str-2}, Connelly, Shen and Smith \cite{con-she-smi}, Malestein and Theran \cite{mal-the}, Owen and Power \cite{owe-pow-crystal}, Power \cite{pow-poly, pow-affine} and Ross, Schulze and Whiteley \cite{ros-sch-whi}.

There is also extensive literature in condensed matter physics concerning the nature and multiplicity of low energy oscillations and rigid unit modes (RUMs) for material crystals in three dimensions. In this case  Bloch's theorem applies and the excitation modes are periodic modulo a phase factor. The set of  phase factors, or, equivalently, the set
of reduced wave vectors for the modes, provides what may be viewed as the RUM spectrum of the crystal.
See Dove et al  \cite{dov-exotic}, Giddy et al \cite{gid-et-al} and Wegner \cite{weg} for example. 
In Owen and Power \cite{owe-pow-crystal} and Power \cite{pow-poly}  the RUM spectrum was formalised in mathematical terms as a subset $\Omega(\C)$ of the $d$-torus $\bT^d$, or as an equivalent subset of $[0,1)^d$, which arises from a choice of translation group $\T$. This set of multi-phases is determined by a matrix-valued function $\Phi_\C(z)$ on $\bT^d$ with the value at  $z=\hat{1} =(1,1,\dots ,1)$ providing the corresponding periodic rigidity matrix for $\T$.

In the present article  we move beyond periodicity and consider infinitesimal flexes of a crystallographic bar-joint framework
which are almost periodic in the classical sense of Bohr. Such flexes are independent of any choice of translation group and so are intrinsic to $\C$ as an infinite bar-joint framework.
It is shown that $\C$ is  almost periodically infinitesimally rigid if and only if for some choice of translation group it 
is periodically infinitesimally rigid and the corresponding RUM spectrum is the minimal set $\{\hat{1}\}$.
More generally, we show how the almost periodic infinitesimal flexes of $\C$ are determined in terms of the matrix function $\Phi_\C(z)$. 

An ongoing interest in the analysis of low energy modes in material science is to quantify the implications of symmetry and local geometry for the set of RUM wave vectors. See for example Kapco et al \cite{kap-daw-riv-tre} where the phenomenon of extensible flexibility is related to the maximal symmetry and minimal density forms of an idealised  zeolite crystal framework. Here the term extensive flexibility corresponds to a maximal rigid unit mode spectrum $\Omega(\C)=\bT^3$. We show that for crystal frameworks whose RUM spectrum decomposes into a finite union of linear components there is a corresponding vector space decomposition of the almost periodic flex space. The flexes in these subspaces are periodic in specific directions associated with certain symmetries of the crystallographic point group. 
 
In Section \ref{Examples} we give a small gallery of crystal frameworks which display a variety of  periodic and almost periodic flexibility properties.

\section{Crystal frameworks and the RUM spectrum}

A bar-joint framework in the Euclidean space $\bR^d$ is a pair consisting of a simple undirected graph $G=(V,E)$ and an injective map $p:V\to \bR^d$.
A  \emph{(real) infinitesimal flex} of $(G,p)$ is a field of velocities, or velocity vectors, $u(v)$ assigned to the joints $p(v)$ such that for every edge $vw\in E$,
\[
(p(v)-p(w))\cdot (u(v)-u(w))=0
\]
If the above condition holds for all pairs $v,w\in V$ then $u$ is a {\em trivial} infinitesimal flex of $(G,p)$.
For convenience we let $p(E)$ denote the set of open line segments $(p(v),p(w))$ with $vw\in E$.
\begin{Definition}
A {\em crystal framework} $\mathcal{C}$ is a bar-joint framework $(G,p)$ for which 
there exist finite subsets $F_v\subseteq p(V)$ and $F_e\subseteq p(E)$ and a full rank translation group $\T$  such that 
\[p(V)=\{T(p_v):p_v\in F_v, \,T\in\T\}\]
\[p(E)=\{T(p_e):p_e\in F_e,\, T\in\T\}\]
\end{Definition}

The pair $(F_v,F_e)$ is called a {\em motif} for $\C$.
The elements of $F_v$ are called {\em motif vertices} and the elements of $F_e$ are called {\em motif edges}. 
The translation group $\T$ is necessarily of the form
\[\T = \{\sum_{j=1}^dk_ja_j: k_j\in \bZ\}\]
where $a_1,a_2,\ldots,a_d$ are linearly independent vectors in $\bR^d$.
The translation $x\mapsto x+\sum_{j=1}^dk_ja_j$ is denoted $T^k$ for each $k=(k_1,\ldots,k_d)\in \bZ^d$.
For each motif vertex $p(v)\in F_v$ and each $k\in\bZ^d$ we denote by $(v,k)$ the unique vertex for which $p(v,k)=T^k(p(v))$. For each motif edge $p_e\in F_e$ with $p_e=(p(v,l),p(w,m))$ we let $(e,k)$ denote the unique edge for which $p_{(e,k)} = ( p(v,l+k),p(w,m+k))$.

A velocity vector or flex $u$ for $\C$ is said to be \emph{strictly periodic} if $u(v,k)=u(v,0)$ for all $k\in \bZ^d$.
In the consideration of infinitesimal rigidity relative to general periodic flexes, or almost periodic flexes, it is convenient and natural to consider complex velocity vectors 
$u:F_v\times \bZ^d \to \bC^d$. Indeed, such vectors are infinitesimal flexes if and only if their real and imaginary parts are infinitesimal flexes.

\subsection{The symbol function and rigidity matrix}

We now define the symbol function $\Phi_\C(z)$ of a crystal framework and the rigidity matrix $R(\C)$ from which it is derived. For $k\in \bZ^d$ the associated monomial function $\bT^d \to \bC$ is written simply as $z^k$.

\begin{Definition}\label{d:Phi2}
Let $\C$ be a crystal framework in $\bR^d$ with  motif 
$(F_v, F_e)$ and for $e=vw\in F_e$ let $p(e)=p(v)-p(w)$. Then $\Phi_\C(z)$ is a matrix-valued function on $\bT^d$ whose rows are labelled by the edges of $F_e$ and whose columns are labelled by the vertex-coordinate pairs in $F_v\times \{1,\dots ,d\}$.
The row for an edge $e= (v,k)(w,l)$ with $v\neq w$ takes the form
\[\kbordermatrix{& & & & v & & & & w & & & \\
e & 0 & \cdots &0 & p(e)\overline{z}^k &0& \cdots&0 &- p(e)\overline{z}^l &0& \cdots &0 }\]
while if $v=w$ it takes the form
\[\kbordermatrix{& & & & v & & & \\
e & 0 & \cdots &0 & p(e)(\overline{z}^k - \overline{z}^l) &0& \cdots&0 }
\]
\end{Definition}
\medskip

\begin{Definition}\label{d:Phi2}
Let $\C$ be a crystal framework in $\bR^d$ with  motif 
$(F_v, F_e)$ and for $e=vw\in F_e$ let $p(e)=p(v)-p(w)$. Then $R(\C)$ is the matrix whose rows are labelled by the edges $(e,k)$ in  $F_e\times \bZ^d$ and whose columns are labelled by the pairs  $(v,k)$ in $F_v\times \bZ^d$.
The row for an edge $(e,k)=(v,l+k)(w,m+k)$, with $e=(v,l)(w,m) \in F_e$  takes the form
\[\kbordermatrix{& & & & (v,l+k) & & & & (w,m+k) & & & \\
(e,k) & \cdots & \cdots &0 & p(e) &0& \cdots&0 &- p(e) &0& \cdots &\cdots }\]
\end{Definition}
\medskip

It follows from this definition  that a velocity vector  $u:F_v\times \bZ^d \to \bR^d$ is an infinitesimal flex for $\C$ if and only if $R(\C)u=0$.

\subsection{The RUM spectrum}

A (real or complex) velocity vector $u$ is said to be \emph{phase-periodic}, or, more precisely, $\omega$-phase-periodic for the fixed multi-phase $\omega\in \bT^d$, if $u(v,k)=\omega^k u(v,0)$ for all  $v \in F_v$, $k\in \bZ^d$. Here  $\omega^k$ is the product $\omega_1^{k_1} \dots \omega_d^{k_d}$. We also write $u = b \otimes e_\omega$ for this vector,
where $b$ is the vector $(u(v,0))_v$ in $\bC^{d|F_v|}$ and $e_\omega$ is the
multi-sequence $(\omega^k)_{k\in\bZ^d}$. The following theorem is given in
\cite{pow-poly} and \cite{pow-seville}.

\begin{Theorem}
Let $\C$ be a crystal framework in $\bR^d$, let $\omega$ be a multi-phase
in $\bT^d$ and let  $u=b\otimes e_\omega$, with $b\in F_v\times \bC^d$, be a  $\omega$-phase-periodic velocity field. Then the following conditions are equivalent.
\begin{enumerate}[(i)]
\item $R(\C)u=0$.
\item $\Phi(\overline{\omega})b=0.$
\end{enumerate}
\end{Theorem}

\begin{Corollary} Let  $\C$ be  a crystal framework in $\bR^d$ with translation group $\T$. Then the following statements are equivalent.
\begin{enumerate}[(i)]
\item The $\T$-periodic real infinitesimal flexes of $\C$ are trivial. 
\item The $\T$-periodic complex infinitesimal flexes of $\C$ are trivial. 
\item The periodic rigidity matrix $\Phi_\C(1,\dots ,1)$ has rank equal to $d|F_v|-d$.
\end{enumerate}
\end{Corollary}

The {\em RUM spectrum} $\Omega(\C)$ of $\C$ is defined to be the set of multi-phases $\omega$ for which there exists a nonzero phase-periodic infinitesimal flex for $\C$ as an infinite-bar-joint framework, or, equivalently, as the set of multi-phases for which the rank of the matrix $\Phi_\C(\overline{\omega})$ is less than
$d|F_v|$.

\section{Almost periodic rigidity}

In this section we first outline the  proof of the fundamental approximation theorem for uniformly almost periodic functions and its counterpart for almost periodic sequences. These ensure that a function (or sequence) which is almost periodic in the sense of Bohr, is approximable by trigonometric functions (or sequences)
that are obtained in an explicit manner from convolution with Bochner-Fej\'er kernels. 
A convenient self-contained exposition of this fact for function approximation, due to Besicovitch \cite{bes-boh}, is given in Partington \cite{par-book}. The direct arguments there can be extended to almost periodic vector-valued functions on $\bZ^d$  and this embraces the setting of velocity fields relevant to the almost periodic rigidity of crystal frameworks. (See Definition \ref{apdef}.) The constructive approximation theorem that we require is given in Theorem \ref{t:APZapprox}.  This theorem together with Lemmas \ref{L1} and \ref{L2} lead to the almost periodic rigidity theorem.

\subsection{Almost periodic sequences}

First we recall the classical theory for univariable functions on $\bR$. 
The Fej\'er kernel functions are given by 
\[
K_n(x)=\sum_{|m|\leq n+1}(1-\frac{|m|}{n+1})e^{imx}, \quad x \in \bR
\]
The positivity of the $K_n$ and their approximate identity property under convolution with continuous periodic functions feature in a standard proof that a continuous $2\pi$-periodic function $f(x)$ on the real line is uniformly approximable by the explicit trigonometric functions 
\[
g_n(x) = \int_0^{2\pi} f(s)K_n(s-x)\frac{ds}{2\pi}
\] 
Almost periodic functions on the real line in the sense of Bohr are similarly uniformly approximable
by an explicit sequence of trigonometric polynomials that are determined by convolution with certain Bochner-Fej\'er kernels.

Note first that it is elementary that the functions $g_n(x)$ have the form
$g_n(x) = [f,R_xK_n]$ where $[\cdot,\cdot]$ is the mean inner product
\[ [f_1,f_2]= \lim_{T\to \infty} \frac{1}{2T}\int_{-T}^T f_1(s)\overline{f_2(s)}ds
\]
and $R_xK_n(s)=K_n(s-x)$.
It is classical that for a function $f(x)$ in $AP(\bR,\bC)$ there is a sequence
of Bochner-Fej\'er kernels $K_n'$  which provide, by the same formula, a uniformly approximating sequence of trigonometric polynomials $g_n(x)$. In this case  the frequencies $\lambda$ in the nonzero terms $ae^{i\lambda x}$ of these approximants appear in
a countable set derived (by rational division) from the spectrum $\Lambda(f)$ of $f$ defined by
\[
\Lambda(f) = \{\lambda \in \bR: [f(x),e^{i\lambda x}] \neq 0\}
\]
It follows from a Parceval inequality for almost periodic functions that this spectrum, which we refer to as the Bohr spectrum, is a well-defined finite or countable set.  

Similar considerations apply to the space $AP(\bZ,\bC)$ of almost periodic sequences. The approximants are general  trigonometric sequences, that is, sequences $(h(k))_{k\in \bZ}$ that have  a finite sum  form
\[
h(k) = \sum_{\omega=e^{i \lambda}:\lambda \in F} a_\lambda \omega^k
\]
so that, in our earlier notation,
\[
h=\sum_{\omega=e^{i \lambda}:\lambda \in F} a_\lambda e_\omega
\]
where $F$ is a finite subset of $\bR$.
The Bohr spectrum of a sequence $h$ in $AP(\bZ,\bC)$ is defined to be the set
\[
\Lambda (h) = \{\lambda \in \bR: [h,e_\omega]_\bZ \neq 0, \mbox{ for } \omega=e^{i\lambda}\}
\]
where
\[
[h_1,h_2]_\bZ= \lim_{N\to \infty}\frac{1}{2N+1}\sum_{|k|\leq N} h_1(k)\overline{h_2(k)}
\]
This spectrum of $\omega$ values is now a subset of $\bT$.
Partial counterparts of the Fej\'er kernel functions $K_n(x)$ are given by the Fej\'er sequences 
\[
K_{(n,\lambda)}= \sum_{|m|\leq n+1}(1-\frac{|m|}{n+1})e_{\omega^m}
\]
associated with a single frequency $\omega = e^{i\lambda}$.
In particular
\[
K_{(n,\lambda)}(k) = \sum_{|m|\leq n+1}(1-\frac{|m|}{n+1}){\omega^{mk}}
\]

The following fundamental approximation  theorem indicates the explicit construction of the Bochner-Fej\'er kernels $K_n'$ for $h$ as coordinate-wise products of appropriate Fej\'er sequences. 

\begin{Theorem}\label{t:APZapprox}
Let $h$ be a sequence in $AP(\bZ,\bC)$, let $\alpha_1, \alpha_2,\dots $ be a maximal subset of the Bohr spectrum $\Lambda(h)$ which is independent over $\bQ$,  and for $n=1,2, \dots $ let
\[
K_n' = \prod_{k=1}^n K_{(n.n!-1,\alpha_k/n!)}
\]
Then $h$ is the uniform limit of the sequence $g_1, g_2, \dots $ of trigonometric sequences in $AP(\bZ,\bC)$ given by
\[
g_n(k) = [h,R_kK_n']_\bZ, \quad k\in \bZ
\]
\end{Theorem}

The main ingredient in the proof of the theorem is that the Bohr spectrum is nonempty if $h\neq 0$ and the arguments for this depend on the equivalence of Bohr almost periodicity with the Bochner condition that the set of translates of $h$ is precompact for the uniform norm. (See \cite{par-book}.)

The arguments leading to Theorem \ref{t:APZapprox} can be generalised to obtain an exact counterpart theorem for $AP(\bZ^d,\bC^r)$. 
The approximating trigonometric sequences $g$ now  have a finite sum  form
\[
g = \sum_{\omega \in F \subset \bT^d} a_\omega \otimes e_\omega
\]
where $a_\omega \in \bC^r$ and $e_\omega$, for $\omega =(\omega_1, \dots ,\omega_d)$ in $\bT^d$, is the pure frequency sequence $e_{\omega_1}\otimes \dots \otimes e_{\omega_d}$ in $AP(\bZ^d,\bC)$ with
\[
e_\omega(k)=\omega_1^{k_1}\dots\omega_d^{k_d}
\] 
The Bohr spectrum $\Lambda(h)$ is similarly defined and is a countable subset of points $\omega$ in $\bT^d$ which we freely identify with a countable subset
of points $\lambda$ in $[0,2\pi)^d$. For notational convenience we state the general theorem only in the case $d=2$.

The metric which is appropriate in our context for the approximation of velocity fields is the uniform metric or norm
 $\|\cdot\|_\infty$; for velocity fields $h, h'$ we have
\[
\|h -h'\|_\infty = \sup_{k,\kappa} \{\|h(k,\kappa) - h'(k,\kappa)\|_2 \}
\]

\begin{Theorem}
\label{MultiApprox}
Let $h$ be a sequence in $AP(\bZ^2,\bC^r)$, let $(\alpha_1,\beta_1), (\alpha_2, \beta_2), \dots $ be a maximal subset of $\Lambda(h) \subset [0,2\pi)^2$ which is independent over $\bQ$,  and for $n=1,2, \dots $ let
\[
K_n^{(2)} = (\prod_{k=1}^n K_{(n.n!-1,\alpha_k/n!)})
(\prod_{k=1}^n K_{(n.n!-1,\beta_k/n!)}) \in AP(\bZ^2,\bC)
\]
Then $h$ is the uniform limit of the sequence $g_1, g_2, \dots $ of trigonometric sequences in  $AP(\bZ^2,\bC^r)$ given by
\[
g_n(k) = [h,R_kK_n^{(2)}]_{\bZ^2}, \quad k\in \bZ^2
\]
\end{Theorem}

Note that here $[\cdot,\cdot]_{\bZ^2}$ is the natural well-defined sesquilinear map from $ AP(\bZ^2,\bC^r)\times  AP(\bZ^2,\bC)$ to $\bC^r$.

We remark that the theory of uniformly almost periodic multi-variable functions on $\bR^d$ and on $\bZ^d$ is part of the abstract theory of almost periodic functions on locally compact abelian groups due to Bochner and von Neumann \cite{boc-von}. For further detail see also Levitan and Zhikov \cite{lev-zhi}, Loomis \cite{loo-book} and Shubin \cite{shu}.

\subsection{Almost periodic rigidity}

We now characterise when a crystal framework admits no nontrivial almost periodic infinitesimal flexes  and in this case we say that it is almost periodically rigid. This is evidently a form of rigidity which is independent of any choice of translation group.

For a crystal framework $\C$ with full rank translation group $\T$ various linear transformations  may be associated with the rigidity matrix $R(\C)$. These transformations are restrictions
of the induced linear transformation
\[
R(\C): \mathbb{C}^{\mathbb{Z}^d\times |F_v|} \otimes \mathbb{C}^{d}\to \mathbb{C}^{\mathbb{Z}^d\times |F_e|} \otimes \mathbb{C}
\]
where  $ \mathbb{C}^{\mathbb{Z}^d\times |F_v|} \otimes \mathbb{C}^{d}$
is the vector space of all velocity fields on $p(V)$, that is, the space of functions $h:{\mathbb{Z}^d\times |F_v|} \to \bC^d$. The codomain
of $R(\C)$
is the vector space of complex-valued functions on the set of edges.

The right shift operators on the domain and codomain of $R(\C)$ for the integral vector $l$ in $\bZ^d$ are denoted by $R^V_l$ and $R^E_l$ respectively. Here the right shift of a sequence $h(k,\kappa)$ by $l$ is the sequence $h(k-l,\kappa)$. We note that 
\[R(\C)\circ R^V_l = R^E_l\circ R(\C)\]

\begin{Definition}\label{apdef}
Let  $h:{\mathbb{Z}^d\times F_v} \to \bC^d$ be a velocity field.
\begin{enumerate}
\item An integral vector $l$ in $\bZ^d$ is an $\epsilon$-translation vector for $h$ if $\|R^V_l(h)-h\|_\infty < \epsilon$.
\item The velocity field $h$ is Bohr almost periodic if for every $\epsilon >0$ the set of $\epsilon$-translation vectors $l$ is relatively dense in $\bR^d$.
\end{enumerate}
\end{Definition}

\begin{Lemma}\label{L1} Let $g$ be the vector-valued trigonometric  multi-sequence 
with finite sum representation
\[
g= \sum_{\omega \in F \subset \bT^d} a_\omega \otimes e_\omega
\] 
with nonzero coefficients $a_\omega$ in $\bC^{d|F_v|}$. If $g$ is a nonzero infinitesimal flex for $\C$ then each component sequence
$a_\omega \otimes e_\omega$ is a nonzero $\omega$-phase-periodic infinitesimal flex. 
\end{Lemma}

\begin{proof}
For $\omega = (\omega_1,\dots ,\omega_d) \in \bT^d$ and $N\in \bN$ let $R^V(\omega, N)$  be the linear map on the normed space $\ell^\infty(\bZ^d\times F_v, \bC^d)$ given in terms of the right shift operators $R^V_k$, $k\in \bZ^d$, by
\[
R^V(\omega,N)=\frac{1}{(N+1)^d}\sum_{k:0\leq k_i\leq N} \bar{\omega}^kR^V_{-k}
\]
Similarly, let $R^E(\omega,N)$ be the linear map on $\ell^\infty(\bZ^d\times F_e, \bC)$ given by,
\[
R^E(\omega,N)=\frac{1}{(N+1)^d}\sum_{k:0\leq k_i\leq N} \overline{\omega}^kR^E_{-k}
\]
The sequence $R(\omega',N)(a_\omega\otimes e_\omega)$
converges uniformly to $a_\omega\otimes e_\omega$ when $\omega' = \omega$ and to the zero sequence otherwise since for each $(l,v_\kappa)\in\bZ^d\times F_v$,
\begin{eqnarray*}
\lim_{N\to\infty} R^V(\omega',N)(a_\omega\otimes e_\omega)(l,v_\kappa) 
&=& \lim_{N\to\infty} \frac{1}{(N+1)^d}\sum_{k:0\leq k_i\leq N} \overline{\omega'}^kR^V_{-k}(a_\omega\otimes e_w)(l,v_\kappa) \\ 
&=&\lim_{N\to\infty} \frac{1}{(N+1)^d}\sum_{k:0\leq k_i\leq N} \overline{\omega'}^k\omega^{l+k}a_\omega\\
&=&\left( \lim_{N\to\infty} \frac{1}{(N+1)^d}\sum_{k:0\leq k_i\leq N} \overline{\omega'}^k\omega^{k}\right)\omega^la_\omega
\end{eqnarray*}
and 
\[ \lim_{N\to\infty} \frac{1}{(N+1)^d}\sum_{k:0\leq k_i\leq N} \overline{\omega'}^k\omega^{k}
=\left\{\begin{array}{ll}
1& \mbox{ if }\omega=\omega' \\
0 & \mbox{ if }\omega\not=\omega' \\
\end{array}\right.\]
Note also that $R(\C)$ is a bounded linear transformation from
$\ell^\infty(\bZ^d\times F_v,\bC^d)$ to $\ell^\infty(\bZ^d\times F_e,\bC)$ which commutes with the right shift operators. Thus
 if $\omega$ is a multi-frequency for $g$ then
\[
R(\C)(a_{\omega}\otimes e_{\omega})  = \lim_{N\to \infty}R(\C)(R^V(\omega,N)g) =\lim_{N\to \infty}R^E(\omega,N)R(\C)g = 0
\] 
\end{proof}

\begin{Lemma}\label{L2}
Let $K$ be a trigonometric polynomial in $AP(\bZ^d\times F_v,\bC)$ and let $h$ be an infinitesimal flex for $\C$ in $AP(\bZ^d\times F_v,\bC^{d})$. Then the mean convolution multi-sequence $g:\bZ^d\times F_v\to\bC^d$ given by $g(k,v_\kappa)=[h,R_k(K)]_{\bZ^d}$ is an infinitesimal flex  for $\C$.
\end{Lemma}

\begin{proof}
By linearity it suffices to assume that $K$ is the elementary multi-sequence $e_\omega$, so that $K(k)=\omega^k$ for $k\in \bZ^d$. Then $g$ is the uniform limit of the sequence $(g_N)$ where
\[
g_N(k,v_\kappa) = \frac{1}{(N+1)^d} \sum_{1\leq s_i\leq N} \bar{\omega}^s(R_{-s}h)(k,v_\kappa)
\]
To see this note that the convergence is uniform if $h$ is a trigonometric sequence. Since the linear maps $h \to g_N$ are contractive for the uniform norm, uniform convergence holds for a general almost periodic velocity sequence. Thus, since the vector space of infinitesimal flexes is invariant under translation it follows that $R(\C)g_N=0$ for each $N$ and hence that $R(\C)g=0$.
\end{proof}

The following theorem shows that a crystal framework is almost periodically rigid if and
only if it is periodically rigid and the RUM spectrum is trivial.

\begin{Theorem}
\label{APThm}
Let $\C$ be a crystallographic bar-joint  framework in $\bR^d$.
The following statements are equivalent.
\begin{enumerate}[(i)]
\item
Every almost periodic infinitesimal flex of $\C$ is trivial.
\item
Every strictly periodic infinitesimal flex of $\C$ is trivial and  $\Omega(\C)=\{\hat{1}\}$. 
\end{enumerate}
\end{Theorem}

\proof
$(i)\Rightarrow (ii)$
This follows since every phase periodic infinitesimal flex is also an almost periodic infinitesimal flex.

$(ii)\Rightarrow (i)$
Let $u$ be an almost periodic infinitesimal flex. Then by Theorem \ref{MultiApprox},
for $d$ dimensions, $u$ is a uniform limit of the sequence $(g_n)$ of trigonometric sequences in
$AP(\bZ^d\times F_v,\bC^d)$ given by
\[g_n(k,v_\kappa) = [h,R_k(K^{(d)}_n)]_{\bZ_d} ,\,\,\,\,\,\,\, k \in \bZ^d\]
where $K^{(d)}_n$, $n = 1, 2,\ldots$, is the sequence of Bochner-Fej\'{e}r kernels for $u$. By Lemma \ref{L2} 
each trigonometric sequence  $g_n$ is an infinitesimal flex of $\C$ and so by Lemma \ref{L1}  each $g_n$ is a finite linear
combination of phase-periodic infinitesimal 
flexes of $\C$. By hypothesis, the RUM spectrum
is trivial and so each $g_n$ is strictly periodic. It follows that $u$ is strictly periodic as
desired.
\endproof

The Bohr spectrum of an almost periodic infinitesimal flex $u$ of the crystal framework $\C$ 
is the finite or countable set given by
\[
\Lambda(u,C) = \{\lambda \in [0,1)^d: [u,e_\omega]_{\bZ^d} \neq 0, \mbox{ for } \omega=e^{2\pi i\lambda}\in \bT^d\}
\]
It follows from Theorem \ref{MultiApprox}, Lemma \ref{L1} and Lemma \ref{L2}, as in the proof above, that $\Lambda(u,C)$ is contained in the RUM spectrum of $\C$. Also, since phase-periodic flexes are almost periodic it follows that the RUM spectrum as a subset of $[0,1)^d$ is the union of the Bohr spectra of all almost periodic infinitesimal flexes.
Note that the spectra here depend on the translation group in the following manner. If $\C'$ has the same underlying bar-joint framework as $\C$ but full rank translation group $\T'\subseteq \T$ then the infinitesimal flex $u$ is represented anew as a sequence in $AP(\bZ^d\times F_v',\bC^d)$ where $F'_v$ is a vertex motif for
$\T'$. The Bohr spectrum $\Lambda(u,\C')$ is then the image of $\Lambda(u,\C)$ under the surjective map $\bT^d \to \bT^d$ induced by the inclusion $\T'\subseteq \T$. This follows the same relationship as that for the RUM spectrum noted in \cite{pow-poly}. It also follows from this that the dimension of $\Omega(\C)$, as a topological space or as an algebraic variety, is independent of the translation group and we refer to this integer, which takes values between $0$ and $d$, as the RUM dimension of $\C$.

We see in the next section that $\Omega(\C)$, as a subset of $[0,1)^d$, often decomposes as a union of linear components. This is the case for example in two dimensions if $\Phi_\C(z)$ is a square matrix function whose determinant polynomial $\det\Phi_\C(z)$ either vanishes identically or factorises into simple factors of the form $(z^n - \lambda w^m)$ with $|\lambda|=1$. It follows that each almost periodic flex $u$ of $\C$ admits a finite sum decomposition  $u_1+\dots +u_r$ in which each component  $u_i$ is an almost periodic flex whose Bohr spectrum lies in the $i^{th}$ linear component. Such component   flexes are partially periodic, being periodic in certain directions of translational symmetry.

\section{Gallery of crystal frameworks} 
\label{Examples}

We now exhibit a number of illustrative examples. The first two of these show two extreme cases, firstly where the RUM spectrum is a singleton, and secondly where the RUM spectrum is $\bT^d$.

\begin{Example}
Let $\C=(G,p)$ be the crystallographic bar-joint framework with motif $(F_v,F_e)$ and translation group $\T$ indicated in Table \ref{CFTriangle}. Simplifying earlier notation, the motif vertex is labelled $v$ and the motif edges are labelled 
$e_0=v(0,0)v(1,0)$, $e_1=v(1,0)v(0,1)$ and $e_2=v(0,0)v(0,1)$.
The translation group is $\T=\{k_1a_1+k_2a_2:k_1,k_2\in \bZ\}$ where 
$a_1=(1,0)$ and $a_2=(\frac{1}{2},\frac{\sqrt{3}}{2})$.
The symbol function for $\C$ is,
\[\Phi_\C(z,w) = \kbordermatrix{ & v_x & & v_y \\
e_0 & \bar{z}-1 & & 0 \\
 e_1 &\frac{1}{2}(\bar{z} - \bar{w}) & & \frac{\sqrt{3}}{2}(\bar{w} - \bar{z})\\
 e_2 & \frac{1}{2}(\bar{w} - 1) & & \frac{\sqrt{3}}{2}(\bar{w}-1)\\
}\]
Note that $\Phi_\C(z,w)$ has rank $2$ unless $z=w=1$ and so the RUM spectrum of 
$\C$ is the singleton $(1,1)\in\bT^2$. 
Also, there are no non-trivial strictly periodic infinitesimal flexes of $\C$ and so, by Theorem \ref{APThm}, $\C$ is almost periodically infinitesimally rigid.

In fact $\C$ is sequentially infinitesimally rigid in the sense that there exists an increasing chain of finite subgraphs $G_1\subset G_2\subset \cdots$ of $G$ such that every vertex of $G$ is contained in some $G_n$ and each subframework  $(G_n,p)$ is infinitesimally rigid. For example, for each $n$ take $G_n$ to be the vertex-induced subgraph on $\{v(k_1,k_2): (k_1,k_2)\in\bZ^2, \,|k_i|\leq n\}$.
It follows that $\C$ admits no nontrivial infinitesimal flexes and so is (absolutely) infinitesimally rigid as a bar-joint framework.
In \cite{kit-pow-1} we obtain a general characterisation of countable simple graphs $G$ whose locally generic placements are infinitesimally rigid in this sense. The condition is that $G$ should contain a vertex-complete chain of $(2,3)$-tight subgraphs. The crystal framework $\C$ may be viewed as a nongeneric placement of such a graph which remains infinitesimally rigid despite the crystallographic symmetry.
\end{Example}

\begin{table}[h]
 \caption{An  infinitesimally rigid crystal framework.}
%  \centering
  \begin{tabular}{ | c | c | c | }
  \hline
  Motif & Translation group & Crystal framework \\ 
  \hline
  \begin{minipage}{.3\textwidth}
  \begin{tikzpicture}[scale=1.6]
  \clip (-1.1,-0.5) rectangle (1.5cm,1.2cm);   
  
  \def\x{0.866};  

  \coordinate (A1) at (0,0);
  \coordinate (A2) at (1,0);
  \coordinate (A3) at (0.5,\x);
 
  \draw (A1) -- (A2) -- (A3) -- cycle;	

  \node[draw,circle,inner sep=2pt,fill] at (A1) {};
  \node[draw,circle,inner sep=2pt,fill=white] at (A2) {};
  \node[draw,circle,inner sep=2pt,fill=white] at (A3) {};

  \node[below left] at (A1) {\small $v$};
  \node[right] at (0.72,0.45) {\small $e_1$};
  \node[below] at (0.5,0) {\small $e_0$};
  \node[left] at (0.28,0.45) {\small $e_2$};
  \end{tikzpicture}
  \end{minipage}
  &
  \begin{minipage}{.3\textwidth}
  \begin{tikzpicture}[scale=1.6] 
  \clip (-0.9,-0.5) rectangle (1.8cm,1.2cm);
      
  \def\x{0.866};  

  \coordinate (A1) at (0,0);
  \coordinate (A2) at (1,0);
  \coordinate (A3) at (0.5,\x);
   
  \draw[line width=1.4pt,->,dashed] (A1) --  (A2);	
  \draw[line width=1.4pt,->,dashed] (A1) -- (A3); 
 
  \node[below] at (A2) {\small $(1,0)$};
  \node[right] at (A3) {\small $(\frac{1}{2},\frac{\sqrt{3}}{2})$};
  \end{tikzpicture}
  \end{minipage}
  & 
  \begin{minipage}{.3\textwidth}
  \begin{tikzpicture}[scale=0.7] 
  \clip (-2.8,-1.3) rectangle (4.4cm,3.8cm); 
  \foreach \b in {-2,-1,0,1,2,3,4}
  {	  
   \foreach \a in {-5, -4,-3, -2, -1, 0, 1, 2, 3, 4,5}
   { 
   \def\x{0.866};  

   \coordinate (A1) at (0+\a+0.5*\b,0+\b*\x);
   \coordinate (A2) at (1+\a+0.5*\b,\b*\x);
   \coordinate (A3) at (0.5+\a+0.5*\b,\x+\b*\x);
 
   \draw (A1) -- (A2) -- (A3) -- cycle;	
   
   \node[draw,circle,inner sep=1pt,fill] at (A1) {};
   \node[draw,circle,inner sep=1pt,fill] at (A2) {};
   \node[draw,circle,inner sep=1pt,fill] at (A3) {};  
   }
  }
  \end{tikzpicture}
  \end{minipage}
  \\ \hline
  \end{tabular}
   \label{CFTriangle}
\end{table}

A crystal framework $\C$ in $\bR^d$ is said to be in {\em Maxwell counting equilibrium} if 
$|F_e|=d|F_v|$. In this case the symbol function $\Phi_\C(z)$ is a square matrix and the determinant of $\Phi_\C(z)$ gives rise to the {\em crystal polynomial} $p_\C(z)$ (see \cite{pow-poly}).
A non-trivial infinitesimal flex $u$ of $(G,p)$ is said to {\em local} if $u(v)=0$ for all but finitely many vertices $v\in V(G)$. It is shown in \cite{owe-pow-crystal} that a crystal framework which is in Maxwell counting equilibrium has a local infinitesimal flex if and only if the crystal polynomial $p_\C(z)$ is identically zero.

\begin{Example}
Consider the crystal framework $\C$ with the motif and translation group shown in Table \ref{CFLocal}.
The motif vertices are $v_0=(0,0)$ and $v_1=(\frac{1}{2},\frac{1}{2})$.
The motif edges are $e_0=v_0(0,0)v_0(1,0)$,
$e_1=v_0(0,0)v_0(0,1)$,
$e_2=v_0(0,1)v_1(0,0)$ and
$e_3=v_0(1,0)v_1(0,0)$.
Note that $\C$ is in Maxwell counting equilibrium and has symbol function,
\[\Phi_\C(z,w)=\kbordermatrix{
 & v_{0,x} & v_{0,y} & & v_{1,x} & v_{1,y} \\
e_0 & \bar{z}-1 & 0 &\vrule & 0 & 0 \\
e_1 & 0 & \bar{w}-1 & \vrule & 0 & 0 \\
e_2 & -\frac{1}{2}\bar{w} & \frac{1}{2}\bar{w} & \vrule & \frac{1}{2} & -\frac{1}{2} \\
e_3 & \frac{1}{2}\bar{z} & -\frac{1}{2}\bar{z} & \vrule & -\frac{1}{2} & \frac{1}{2} 
}\]
The determinant of $\Phi_\C(z,w)$ vanishes identically and so the RUM spectrum of $\C$ is $\bT^2$.
A local infinitesimal  flex of $\C$ is evident by defining $u(v_1)=(1,1)$ and $u_v=0$ for all $v\not=v_1$. 
A phase-periodic infinitesimal flex of $\C$ for $\omega=(\omega_1,\omega_2)$ is obtained by taking 
$u(v_0,k)=0$ and $u(v_1,k)=\omega_1^{k_1}\omega_2^{k_2}(1,1)$ for each $k=(k_1,k_2)\in\bZ^2$. In particular, any finite linear combination of such phase-periodic flexes will be an almost periodic infinitesimal flex for $\C$.
\end{Example}

\begin{table}[ht]
  \centering
  \begin{tabular}{ | c | c | c | }

  \hline
  Motif & Translation group & Crystal framework \\ 
  \hline
    
  \begin{minipage}{.3\textwidth}
  \begin{tikzpicture}[scale=2] 
  \clip (-0.85,-0.2) rectangle (1.8cm,1.2cm); 
  
  \coordinate (A1) at (0,0);
  \coordinate (A2) at (1,0);
  \coordinate (A3) at (0.5,0.5);
  \coordinate (A4) at (0,1);

  \draw (A1) -- (A2) -- (A3) -- (A4) --  cycle;	

  \node[draw,circle,inner sep=2pt,fill] at (A1) {};
  \node[draw,circle,inner sep=2pt,fill=white] at (A2) {};
  \node[draw,circle,inner sep=2pt,fill] at (A3) {};
  \node[draw,circle,inner sep=2pt,fill=white] at (A4) {};

  \node[left] at (A1) {\small $v_0$};
  \node[right] at (A3) {\small $v_1$};
  \node[below] at (0.5,0) {\small $e_0$};
  \node[left] at (0,0.5) {\small $e_1$};
  \node[right] at (0.24,0.76) {\small $e_2$};
  \node[right] at (0.75,0.25) {\small $e_3$};

  \end{tikzpicture}
  \end{minipage}
   &
  \begin{minipage}{.3\textwidth}
  \begin{tikzpicture}[scale=2] 
  \clip (-0.8,-0.25) rectangle (2cm,1.2cm); 
  
  \coordinate (A1) at (0,0);
  \coordinate (A2) at (1,0);
  \coordinate (A4) at (0,1);

  \draw[line width=1.2pt,->,dashed] (A1) --  (A2);	
  \draw[line width=1.2pt,->,dashed] (A1) -- (A4);
  \node[left] at (A4) {\small $(0,1)$};
  \node[below] at (A2) {\small $(1,0)$};

  \end{tikzpicture}
  \end{minipage}
  & 
  \begin{minipage}{.3\textwidth}
  \vspace{3mm}
  \begin{tikzpicture}[scale=0.8] 
  \clip (-3.4,-1.7) rectangle (2.9cm,2.8cm); 
 
  \foreach \b in {-2,-1,0,1,2}
  {
	  
   \foreach \a in {-4,-3,-2,-1,0,1,2,3}
   {
   
   \coordinate (A1) at (0+\a,0+\b);
   \coordinate (A2) at (1+\a,\b);
   \coordinate (A3) at (0.5+\a,0.5+\b);
   \coordinate (A4) at (\a,1+\b);

   \draw (A1) --  (A2) -- (A3) -- (A4) -- cycle;	

   \node[draw,circle,inner sep=1pt,fill] at (A1) {};
   \node[draw,circle,inner sep=1pt,fill] at (A2) {};
   \node[draw,circle,inner sep=1pt,fill] at (A3) {};
   \node[draw,circle,inner sep=1pt,fill] at (A4) {};

   }
  }

  \end{tikzpicture}
  \end{minipage} \\ 
  \hline
  \end{tabular}
  \caption{A  crystal framework with full RUM spectrum.}
  \label{CFLocal}
\end{table}

A  velocity field $u$ is {\em supercell-periodic} for a crystal framework $\C$  if $u(v_{\kappa},0)=u(v_\kappa,k)$ for each motif vertex $v_\kappa$ and all $k$ in a
full rank subgroup of $\bZ^d$.

\begin{Example}
Let $\C$ be the crystallographic bar-joint framework with motif $(F_v,F_e)$ and translation group $\T$ indicated in Table \ref{AltGrid}. 
Note that $\C$ has symbol function,
\[\Phi_\C(z,w)=\kbordermatrix{
 & v_{0,x} & v_{0,y} & & v_{1,x} & v_{1,y} \\
e_0 & -1 & 0 &\vrule & \bar{z}w & 0 \\
e_1 & 0 & -1 & \vrule & 0 & 1 \\
e_2 & \bar{z}-1& \bar{z}-1 & \vrule & 0 & 0 \\
e_3 & \bar{z} & 0 & \vrule & -1 & 0\\
e_4 & 0 & \bar{w} &\vrule& 0 & -1
}\]
The RUM spectrum of $\C$ is $\Omega(\C)=\{(1,1),(-1,1)\}$. Note that $\C$ does not admit any non-trivial  infinitesimal flexes which are strictly periodic with respect to $\T$. However, $\C$ does admit non-trivial  supercell-periodic infinitesimal  flexes for the subgroup $2\bZ\times \bZ$.
\end{Example}

\begin{table}[ht]
  \centering
  \begin{tabular}{ | c | c | c | }
  \hline
  Motif & Translation group & Framework \\ 
  \hline
    
  \begin{minipage}{.3\textwidth}
  \begin{tikzpicture}[scale=1.2] 
 
  \clip (-1.4,-0.3) rectangle (3cm,2.3cm); 
  
  \coordinate (A1) at (0,0);
  \coordinate (A2) at (1,0);
  \coordinate (A3) at (0,1);
  \coordinate (A4) at (0,2);
  \coordinate (A5) at (1,1);

  \draw (A1) --  (A2);	
  \draw (A1) -- (A5) -- (A3);
  \draw (A1) --  (A3) -- (A4);	
    
  \node[draw,circle,inner sep=2.2pt,fill] at (A1) {};
  \node[draw,circle,inner sep=2.2pt,fill=white] at (A2) {};
  \node[draw,circle,inner sep=2.2pt,fill] at (A3) {};
  \node[draw,circle,inner sep=2.2pt,fill=white] at (A4) {};
  \node[draw,circle,inner sep=2.2pt,fill=white] at (A5) {};

  \node[below left] at (A1) {\small $v_0$};
  \node[left] at (A3) {\small $v_1$};
  \node[below] at (0.5,0) {\small $e_0$};
  \node[left] at (0,0.5) {\small $e_1$};
  \node[right] at (0.4,0.35) {\small $e_2$};
  \node[above] at (0.5,1) {\small $e_3$};
  \node[left] at (0,1.5) {\small $e_4$};
  
  \end{tikzpicture}
  \end{minipage}
  &
  \begin{minipage}{.3\textwidth}
  \begin{tikzpicture}[scale=1.2] 
 
  \clip (-1.4,-0.3) rectangle (3cm,2.3cm); 
  
  \coordinate (A1) at (0,0);
  \coordinate (A5) at (0,2);
  \coordinate (A6) at (1,1);

  \draw[line width=1.2pt,->,dashed] (A1) --  (A5);	
  \draw[line width=1.2pt,->,dashed] (A1) -- (A6);
  \node[right] at (A6) {\small $(1,1)$};
  \node[right] at (A5) {\small $(0,2)$};
  \end{tikzpicture}
  \end{minipage}  
  & 
  \begin{minipage}{.3\textwidth}
  \begin{tikzpicture}[scale=0.8] 
 
  \clip (-1.6,-1.4) rectangle (4.7cm,3.4cm); 
   
  \foreach \a in {-2,-1,0,1,2,3,4,5,6}
  {	  
   \foreach \b in {-6,-4,-2,0,2,4}
   {
   
   \coordinate (A1) at (0+\a,0+\b+\a);
   \coordinate (A2) at (0+\a,1+\b+\a);
   \coordinate (A3) at (1+\a,0+\b+\a);
   \coordinate (A4) at (1+\a,1+\b+\a);
   \coordinate (A5) at (0+\a,2+\b+\a);

   \draw (A1) -- (A3) -- (A4);	
   \draw (A1) -- (A2) -- (A5);
   \draw (A1) -- (A4);
   \draw (A2) -- (A4);

   \node[draw,circle,inner sep=1.1pt,fill] at (A1) {};
   \node[draw,circle,inner sep=1.1pt,fill] at (A2) {};
   \node[draw,circle,inner sep=1.1pt,fill] at (A3) {};
   \node[draw,circle,inner sep=1.1pt,fill] at (A4) {};
   \node[draw,circle,inner sep=1.1pt,fill] at (A5) {};
    
   }
  }
 
  \end{tikzpicture}
  \end{minipage}
    \\ \hline
  \end{tabular}
  \caption{A crystal framework with RUM spectrum $\Omega(\C)=\{(1,1),(-1,1)\}$.}
\label{AltGrid}
\end{table}

\begin{Example}
Let $\C$ be the crystallographic bar-joint framework in $\bR^2$ which is indicated in Table \ref{DoubleGrid}.
The symbol function $\Phi_\C(z,w)$ is,
\[\Phi_\C(z,w)=\kbordermatrix{
 & v_{0,x} & v_{0,y} & & v_{1,x} & v_{1,y} \\
e_0 & -\frac{1}{2} & \frac{1}{2} &\vrule & \frac{1}{2} & -\frac{1}{2} \\
e_1 & \bar{z}-1& 0 & \vrule & 0 & 0\\
e_2 & 0 & \bar{w}-1 &\vrule& 0 & 0 \\
e_3 & \bar{z}-\bar{w}&\bar{w}-\bar{z} & \vrule &0 &0 \\
e_4 & 0 & 0 & \vrule & \bar{z}-1 & 0 \\
e_5 & 0 & 0 & \vrule & 0 & \bar{w}-1 \\
e_6 &0 &0  & \vrule&\bar{z}-\bar{w} &\bar{w}-\bar{z} 
}\]
The RUM spectrum $\Omega(\C)$  is the singleton $(1,1)\in\bT^2$.
However, every strictly periodic velocity field $u$ with $u(v_0)\not=u(v_1)$ is a non-trivial strictly periodic infinitesimal flex of $\C$.  
\end{Example}

\begin{table}[ht]
  \centering
  \begin{tabular}{ | c | c | c | }
  \hline
  Motif & Translation group & Crystal framework \\ 
  \hline
  \begin{minipage}{.3\textwidth}
  \begin{tikzpicture}[scale=1.9] 
 
  \clip (-0.8,-1) rectangle (2.2cm,1.4cm); 
  
  \coordinate (A1) at (0,0);
  \coordinate (A2) at (1,0);
  \coordinate (A3) at (0,1);
  \coordinate (A4) at (0.35,-0.65);
  \coordinate (A5) at (1.35,-0.65);
  \coordinate (A6) at (0.35,0.35);
   
  \draw (A1) -- (A2);
  \draw (A1) -- (A3);	
  \draw (A1) -- (A4);	
  \draw (A4) -- (A5);
  \draw (A4) -- (A6);	
  \draw (A3) -- (A2);	
  \draw (A5) -- (A6);	
   
  \node[draw,circle,inner sep=2.2pt,fill] at (A1) {};
  \node[draw,circle,inner sep=2.2pt,fill=white] at (A2) {};
  \node[draw,circle,inner sep=2.2pt,fill=white] at (A3) {};
  \node[draw,circle,inner sep=2.2pt,fill] at (A4) {};
  \node[draw,circle,inner sep=2.2pt,fill=white] at (A5) {};
  \node[draw,circle,inner sep=2.2pt,fill=white] at (A6) {};

  \node[left] at (A1) {\small $v_0$};
  \node[left] at (A4) {\small $v_1$};
  \node[left] at (0.2,-0.35) {\small $e_0$};
 
  \end{tikzpicture}
  \end{minipage}
  &
  \begin{minipage}{.3\textwidth}
  \begin{tikzpicture}[scale=1.9] 
  \clip (-0.7,-1) rectangle (2cm,1.4cm); 
  
  \coordinate (A1) at (0,0);
  \coordinate (A2) at (1,0);
  \coordinate (A3) at (0,1);
 
  \draw[line width=1.2,->,dashed] (A1) --  (A2);	
  \draw[line width =1.2,->,dashed] (A1) -- (A3);

  \node[below] at (A2) {\small $(1,0)$};
  \node[below right] at (A3) {\small $(0,1)$};
  \end{tikzpicture}
  \end{minipage}
  & 
  \begin{minipage}{.3\textwidth}
  \begin{tikzpicture}[scale=1.9] 
  \clip (0.6,-0.2) rectangle (3.2cm,2.2cm); 
   
  \foreach \b in {-4,-3,-2,-1,0,1,2,3,4}
  {
  \foreach \a in {-4,-3,-2,-1,0,1,2,3,4,5,6}
   {
   \coordinate (Aone) at (0+\a,0+\b);
   \coordinate (Athree) at (1+\a,0+\b);
   \coordinate (Afive) at (0+\a,1+\b);

   \draw (Aone) --  (Athree);	
   \draw (Aone) -- (Afive);
   \draw (Aone) -- (0.35+\a,-0.65+\b);
   \draw (Athree) -- (Afive);	

   \node[draw,circle,inner sep=1.4pt,fill] at (Aone) {};
   \node[draw,circle,inner sep=1.4pt,fill] at (Athree) {};
   \node[draw,circle,inner sep=1.4pt,fill] at (Afive) {};
   }
  }
 
  \foreach \b in {-4,-3,-2,-1,0,1,2,3,4}
  {
   \foreach \a in {-4,-3,-2,-1,0,1,2,3,4,5,6}
   {
   \coordinate (Aone) at (0.35+\a,-0.65+\b);
   \coordinate (Athree) at (1.35+\a,-0.65+\b);
   \coordinate (Afive) at (0.35+\a,0.35+\b);
 
   \draw (Aone) --  (Athree);	
   \draw (Aone) -- (Afive);
   \draw (Athree) -- (Afive);	
 
   \node[draw,circle,inner sep=1.4pt,fill] at (Aone) {};
   \node[draw,circle,inner sep=1.4pt,fill] at (Athree) {};
   \node[draw,circle,inner sep=1.4pt,fill] at (Afive) {};
   }
  }
  \end{tikzpicture}
  \end{minipage}
  \\ \hline
  \end{tabular}
\caption{A crystal framework with $\Omega(\C)=\{(1,1)\}$ and nontrivial strictly periodic flexes.}
\label{DoubleGrid}
\end{table}

As we have noted, the {\em RUM dimension} of a crystal framework  $\C$ is the dimension of the RUM spectrum $\Omega(\C)$ as a real algebraic variety (see \cite{pow-poly}).

\begin{Example}
The crystal framework $\C$ illustrated in Table \ref{CFRhombic} is in Maxwell counting equilibrium.
The motif vertices are $v_0 = (0,0)$, $v_1=(\frac{1}{2},\frac{\sqrt{3}}{2})$ and $v_2=(0,1)$. 
The  symbol function is, 
\[\Phi(z,w)=\kbordermatrix{
 & v_{0,x} & v_{0,y} &  &  v_{1,x} & v_{1,y} & & v_{2,x} & v_{2,y} \\
e_0 & -\frac{1}{2} & -\frac{\sqrt{3}}{2} & \vrule  & \frac{1}{2} & \frac{\sqrt{3}}{2} & \vrule & 0 & 0 \\
e_1 & 0 & -1 & \vrule & 0 & 0 & \vrule & 0 & 1 \\
 e_2 & \frac{1}{2}\bar{w} & \frac{\sqrt{3}}{2}\bar{w} & \vrule & 0 & 0 & \vrule & -\frac{1}{2} & -\frac{\sqrt{3}}{2} \\
e_3 & 0 & \bar{w} & \vrule & 0 & -1 & \vrule & 0 & 0 \\
e_4 & -\frac{1}{2} &\frac{\sqrt{3}}{2} & \vrule & 0 & 0 & \vrule  & \frac{1}{2}w\bar{z} &- \frac{\sqrt{3}}{2}w\bar{z} \\
e_5 & \frac{1}{2}\bar{z} & -\frac{\sqrt{3}}{2}\bar{z} & \vrule & -\frac{1}{2} & \frac{\sqrt{3}}{2} & \vrule & 0 & 0 
}
\]

The crystal polynomial factors in to linear parts,
\[p_\C(z,w)=(z-1)(w-1)(z-w)\] 
and so the RUM spectrum is a proper subset of $\bT^2$ whose representation in $[0,1)^2$ consists of the points (s,t) in the line segments given by
\[
s=0,\,\,\,\,\,\,\,   t=0,\,\,\,\,\,\,\,  s=t
\]
In particular, $\C$ is almost periodically infinitesimally flexible but has no local infinitesimal flexes.
Also, the RUM dimension of $\C$ is $1$.
It follows also that every almost periodic infinitesimal flex decomposes 
as a sum $u_1+u_2+u_3$ of three almost periodic flexes corresponding to
this ordered decomposition. Furthermore,  $u_1$, with Bohr spectrum in 
the line $s=0$, is periodic in the direction of the period vector 
$a_1=(1,0)$, while $u_2$, with Bohr spectrum in the line $t=0$, is 
periodic in the direction of the  period vector 
$a_2=(1/2,(2+\sqrt{3})/2)$, and $u_3$ with Bohr spectrum in the line 
$s=t$ is periodic in the direction $a_1-a_2$.
\end{Example}

\begin{table}[h]
  \centering
  \begin{tabular}{ | c | c | c | }
  \hline
  Motif & Translation group & Crystal framework \\ 
  \hline
  
  \begin{minipage}{.3\textwidth}
  \begin{tikzpicture}[scale=1.3]
 
  \clip (-1.6,-1.2) rectangle (1.8cm,2.2cm);   
  
  \def\x{0.866};  

  \coordinate (A1) at (0,0);
  \coordinate (A2) at (0.5,\x);
  \coordinate (A3) at (0,1);
  \coordinate (A4) at (0.5,1+\x);
  \coordinate (A5) at (0.5,-\x);
  \coordinate (A6) at (1,0);

  \draw (A1) -- (A2) -- (A4) -- (A3) -- cycle;	
  \draw (A1) -- (A5);
  \draw (A2) -- (A6);

  \node[draw,circle,inner sep=2pt,fill] at (A1) {};
  \node[draw,circle,inner sep=2pt,fill] at (A2) {};
  \node[draw,circle,inner sep=2pt,fill] at (A3) {};
  \node[draw,circle,inner sep=2pt,fill=white] at (A4) {};
  \node[draw,circle,inner sep=2pt,fill=white] at (A5) {};
  \node[draw,circle,inner sep=2pt,fill=white] at (A6) {};

  \node[below left] at (A1) {\small $v_0$};
  \node[right] at (A2) {\small $v_1$};
  \node[left] at (A3) {\small $v_2$};
  \node[below right] at (0.17,0.5*\x+0.06) {\small $e_0$};
  \node[left] at (0.08,0.5) {\small $e_1$};
  \node[above left] at (0.3,1.25) {\small $e_2$};
  \node[right] at (0.45,0.5+\x) {\small $e_3$}; 
  \node[right] at (0.2,-0.5*\x) {\small $e_4$};
  \node[right] at (0.7,0.5*\x) {\small $e_5$}; 
  \end{tikzpicture}

  \end{minipage}
  &
  \begin{minipage}{.3\textwidth}
  \begin{tikzpicture}[scale=1.3] 
 
  \clip (-1.2,-1.2) rectangle (2.2cm,2.2cm);
      
  \def\x{0.866};  

  \coordinate (A1) at (0,0);
  \coordinate (A2) at (1,0);
  \coordinate (A3) at (0.5,1+\x);
   
  \draw[line width=1.2pt,->,dashed] (A1) --  (A2);	
  \draw[line width=1.2pt,->,dashed] (A1) -- (A3); 

  \node[below] at (A2) {\small $(1,0)$};
  \node[below right] at (A3) {\small $(\frac{1}{2},\frac{2+\sqrt{3}}{2})$};
  
  \end{tikzpicture}
  \end{minipage}
  & 
  \begin{minipage}{.3\textwidth}
  \begin{tikzpicture}[scale=0.7] 
 
  \clip (-3.3,-2.2) rectangle (3.9cm,4.1cm); 
  
  \foreach \b in {-3,-2,-1,0,1,2,3}
  {	  
   \foreach \a in {-4, -3, -2, -1, 0, 1, 2, 3, 4, 5}
   {
   \def\x{0.866};  
 
   \coordinate (A1) at (0+\a+1/2*\b,0+\b*1.866);
   \coordinate (A2) at (0.5+\a+1/2*\b,\x+\b*1.866);
   \coordinate (A3) at (0+\a+1/2*\b,1+\b*1.866);
   \coordinate (A4) at (0.5+\a+1/2*\b,1+\x+\b*1.866);
   \coordinate (A5) at (0.5+\a+1/2*\b,-\x+\b*1.866);
   \coordinate (A6) at (1+\a+1/2*\b,0+\b*1.866);
 
   \draw (A1) -- (A2) -- (A4) -- (A3) -- cycle;	
   \draw (A1) -- (A5);
   \draw (A2) -- (A6);
   
   \node[draw,circle,inner sep=1pt,fill] at (A1) {};
   \node[draw,circle,inner sep=1pt,fill] at (A2) {};
   \node[draw,circle,inner sep=1pt,fill] at (A3) {};
   \node[draw,circle,inner sep=1pt,fill] at (A4) {};
   \node[draw,circle,inner sep=1pt,fill] at (A5) {};
   \node[draw,circle,inner sep=1pt,fill] at (A6) {}; 
   }
  }
  \end{tikzpicture}
  \end{minipage}
  \\ \hline
  \end{tabular}
 \caption{A crystal framework with RUM dimension $1$.}
\label{CFRhombic}
\end{table}

\begin{Example} Let $\C$ be the crystal framework illustrated in Table \ref{CFStar}.
The symbol function $\Phi(z,w)$ is the square matrix,
\[ \kbordermatrix{ & v_{1,x} & v_{1,y} & v_{2,x} & v_{2,y} & v_{3,x} & v_{3,y} & v_{4,x} & v_{4,y} & v_{5,x} & v_{5,y} & v_{6,x} & v_{6,y} \\
e_1 & 0 & 0  & 0 & 0 &0 &0 &\bar{z} & \frac{1}{2}\bar{z} &0 &0 &-1 &- \frac{1}{2}\\
e_2 & - \frac{1}{2} & -1 &  0 & 0 & \frac{1}{2} &1 &0 &0 &0 &0 &0 &0 \\
e_3 & 0 & 0 & \frac{1}{2} & 0 & 0 &0 & -\frac{1}{2} &0 &0 &0 &0 &0 \\
e_4 & 0 & 0 & 0 & 0 & - \frac{1}{2} &0 & \frac{1}{2}\bar{z}  &0 &0 &0 &0 &0 \\
e_5 & 0 & 0 & 0 & 0 & 0 &0 &-1 & \frac{1}{2} &0 & 0 &1 &-\frac{1}{2} \\
e_6 & 0 &  \frac{1}{2}\bar{w} & 0 & 0 & 0 &0 &0 &0 &0 & -\frac{1}{2} &0 &0 \\
e_7 &  0 & 0 & 0& 0 & 0 &0 &-1 & -\frac{1}{2} &1 & \frac{1}{2} &0 &0 \\
e_8 & 0 & 0 & 0 & 0 & 0 &0 &\bar{z}  &- \frac{1}{2}\bar{z}  &-1 & \frac{1}{2} &0 &0 \\
e_9 &  \frac{1}{2} \bar{w}& \bar{w} &  -\frac{1}{2} & -1 & 0 &0 &0 &0 &0 &0 &0 &0 \\
e_{10} & - \frac{1}{2}\bar{w} & \bar{w} & 0 & 0 & \frac{1}{2} &-1 &0 &0 &0 &0 &0 &0 \\
e_{11} & 0 &-  \frac{1}{2} & 0 & 0 & 0 &0 &0 &0 &0 &0 &0 & \frac{1}{2}\\
e_{12} &  \frac{1}{2} & -1 & - \frac{1}{2} & 1 & 0 &0 &0 &0 &0 &0 &0 &0 }
\]
The crystal polynomial is
\[p_{\mathcal{C}}(z,w) = (z-1)(z+1)(w-1)(w+1)\]
As in the last example the RUM spectrum decomposes as a union of linear subsets. In the $[0,1)^2$ representation it yields two horizontal and two vertical lines. From this it follows that any almost periodic infinitesimal flex decomposes as a sum of two flexes each of which is supercell periodic in one of the axial directions.
\end{Example}

\begin{table}[ht]
  \centering
  \begin{tabular}{ | c | c | c | }
  \hline
  Motif & Translation group & Framework \\ 
  \hline
  \begin{minipage}{.3\textwidth}
  \begin{tikzpicture}[scale=1.5] 
 
  \clip (-1.7,-1.4) rectangle (1.2cm,1.4cm); 
  
  \coordinate (A1) at (0.5,0);
  \coordinate (A2) at (1,0);
  \coordinate (A3) at (0,0.5);
  \coordinate (A4) at (0,1);
  \coordinate (A5) at (-0.5,0);
  \coordinate (A6) at (-1,0);
  \coordinate (A7) at (0,-0.5);
  \coordinate (A8) at (0,-1);
  
  \draw (A1) -- (A2) -- (A3) -- (A4) -- (A1);	
  \draw (A3) -- (A6) -- (A5);
  \draw (A4) -- (A5) -- (A8) -- (A7);
  \draw (A7) -- (A6);
  \draw (A8) -- (A1);
  \draw (A7) -- (A2);
 
  \node[draw,circle,inner sep=2.2pt,fill] at (A1) {};
  \node[draw,circle,inner sep=2.2pt,fill=white] at (A2) {};
  \node[draw,circle,inner sep=2.2pt,fill] at (A3) {};
  \node[draw,circle,inner sep=2.2pt,fill=white] at (A4) {};
  \node[draw,circle,inner sep=2.2pt,fill] at (A5) {};
  \node[draw,circle,inner sep=2.2pt,fill] at (A6) {}; 
  \node[draw,circle,inner sep=2.2pt,fill] at (A7) {};
  \node[draw,circle,inner sep=2.2pt,fill] at (A8) {}; 

  \end{tikzpicture}
  \end{minipage}
 
    &
 
  \begin{minipage}{.3\textwidth}
  \begin{tikzpicture}[scale=1.5] 
 
  \clip (-1.8,-1.5) rectangle (1.2cm,1.5cm); 
  
  \coordinate (A1) at (-1,-1);
  \coordinate (A2) at (1,-1);
  \coordinate (A3) at (-1,1);
 
 \draw[line width=1pt,->,dashed] (A1) --  (A3);	
 \draw[line width=1pt,->,dashed] (A1) -- (A2);
 
  \end{tikzpicture}
  \end{minipage}
    
    & 
 
  \begin{minipage}{.3\textwidth}
   \begin{tikzpicture}[scale=0.6] 
 
  \clip (-2,-2.6) rectangle (6.4cm,5.2cm); 
    
  \foreach \b in {-4,-2,0,2,4,6}
  {
	  
  \foreach \a in {-4,-2,0,2,4,6}
  {
  \coordinate (A2) at (0.5+\a,0+\b);
  \coordinate (A3) at (1+\a,0+\b);
  \coordinate (A4) at (0+\a,0.5+\b);
  \coordinate (A5) at (0+\a,1+\b);
  \coordinate (A6) at (-0.5+\a,0+\b);
  \coordinate (A7) at (-1+\a,0+\b);
  \coordinate (A8) at (0+\a,-0.5+\b);
  \coordinate (A9) at (0+\a,-1+\b);
  
  \draw (A2) -- (A3) -- (A4) -- (A5) -- (A2);	
  \draw (A4) -- (A7) -- (A6);
  \draw (A5) -- (A6) -- (A9) -- (A8);
  \draw (A8) -- (A7);
  \draw (A9) -- (A2);
  \draw (A8) -- (A3);

  \node[draw,circle,inner sep=1pt,fill] at (A2) {};
  \node[draw,circle,inner sep=1pt,fill] at (A3) {};
  \node[draw,circle,inner sep=1pt,fill] at (A4) {};
  \node[draw,circle,inner sep=1pt,fill] at (A5) {};
  \node[draw,circle,inner sep=1pt,fill] at (A6) {};
  \node[draw,circle,inner sep=1pt,fill] at (A7) {};  
  \node[draw,circle,inner sep=1pt,fill] at (A8) {};
  \node[draw,circle,inner sep=1pt,fill] at (A9) {};  
    }
 }

  \end{tikzpicture}
  \end{minipage}
    \\ \hline
  \end{tabular}
  \caption{A crystal framework with supercell periodic flexes.}
\label{CFStar}
\end{table}

The following two examples, shown in Tables 7 and 8, have the same underlying graph and the same crystallographic point group,  the dihedral group $\C_{2v}$. However the former has RUM dimension $1$ and is linearly indecomposable while the latter  has RUM dimension zero. Note that $4$-regular crystal frameworks such as these have square symbol function.

\begin{Example}
\label{Hex}
Let $\C$ be the crystal framework illustrated in Table \ref{CFHex}.
\[ \left[{\small \begin {array}{cccccccccccccccccc} -1&0&1&0&0&0&0&0&0&0&0&0&0&0
&0&0&0&0\\ \noalign{\medskip}0&0&-\frac{1}{2}&-\frac{
\sqrt {3}}{2}&\frac{1}{2}&\frac{
\sqrt {3}}{2}&0&0&0&0&0&0&0&0&0&0&0&0\\ \noalign{\medskip}0&0&0&0&\frac{1}{2}&-\frac{
\sqrt {3}}{2}&-\frac{1}{2}&\frac{
\sqrt {3}}{2}&0&0&0&0&0&0&0&0&0&0\\ \noalign{\medskip}0
&0&0&0&0&0&1&0&-1&0&0&0&0&0&0&0&0&0\\ \noalign{\medskip}0&0&0&0&0&0&0&0
&\frac{1}{2}&\frac{
\sqrt {3}}{2}&-\frac{1}{2}&-\frac{
\sqrt {3}}{2}&0&0&0&0&0&0
\\ \noalign{\medskip}\frac{1}{2}&-\frac{
\sqrt {3}}{2}&0&0&0&0&0&0&0&0&-\frac{1}{2}&\frac{
\sqrt {3}}{2}&0&0&0&0&0&0\\ \noalign{\medskip}-\frac{1}{2}&\frac{
\sqrt {3}}{2}&0&0&0&0&0
&0&0&0&0&0&\frac{1}{2}&-\frac{
\sqrt {3}}{2}&0&0&0&0\\ \noalign{\medskip}0&0&\frac{1}{2}&\frac{
\sqrt {3}}{2}&0&0&0&0&0&0&0&0&-\frac{1}{2}&-\frac{
\sqrt {3}}{2}&0&0&0&0
\\ \noalign{\medskip}0&0&-1&0&0&0&0&0&0&0&0&0&0&0&\bar{z}&0&0&0
\\ \noalign{\medskip}0&0&0&0&-\frac{1}{2}&\frac{
\sqrt {3}}{2}&0&0&0&0&0&0&0&0&\frac{\bar{z}}{2}&-\frac {\sqrt {3}\bar{z}}{2}&0&0\\ \noalign{\medskip}0&0&0&0&
-\frac{1}{2}&-\frac{
\sqrt {3}}{2}&0&0&0&0&0&0&0&0&0&0&\frac{\bar{z}}{2}&\frac {
\sqrt {3}\bar{z}}{2}\\ \noalign{\medskip}0&0&0&0&0&0&-1&0&0&0&0&0&0&0&0&0&\bar{z}&0\\ \noalign{\medskip}0&0&0&0&0&0&\frac{1}{2}&-\frac{
\sqrt {3}}{2}&0&0&0&0&-
\frac{\bar{w}}{2}&\frac{\sqrt {3}\bar{w}}{2}&0&0&0&0\\ \noalign{\medskip}0
&0&0&0&0&0&0&0&-\frac{1}{2}&-\frac{
\sqrt {3}}{2}&0&0&\frac{\bar{w}}{2}&\frac {
\sqrt {3}\bar{w}}{2}&0&0&0&0\\ \noalign{\medskip}0&0&0&0&0&0&0&0&1&0&0&0&0&0
&0&0&-1&0\\ \noalign{\medskip}0&0&0&0&0&0&0&0&0&0&\frac{1}{2}&-\frac{
\sqrt {3}}{2}&0
&0&0&0&-\frac{1}{2}&\frac{
\sqrt {3}}{2}\\ \noalign{\medskip}0&0&0&0&0&0&0&0&0&0&\frac{1}{2}
&\frac{
\sqrt {3}}{2}&0&0&-\frac{1}{2}&-\frac{
\sqrt {3}}{2}&0&0\\ \noalign{\medskip}1&0&0
&0&0&0&0&0&0&0&0&0&0&0&-1&0&0&0\end {array}} \right] 
\]
The crystal polynomial factorizes in to linear factors,
\[p_{\mathcal{C}}(z,w)  = (z+1)(z-1)^3 \]
\end{Example}

\begin{table}[h]
  \centering
  \begin{tabular}{ | c | c | c | }
  \hline
  Motif & Translation group & Framework \\ 
  \hline
  \begin{minipage}{.3\textwidth}
  \begin{tikzpicture}[scale=1]
 
  \clip (-2.5,-2.2) rectangle (2cm,2.2cm); 
  
  \def\a{0.5};
  \def\b{0.866};
  \def\c{0.866};
  \def\d{0.5};

  \coordinate (A1) at (-0.5,-\b);
  \coordinate (A2) at (0.5,-\b);
  \coordinate (A3) at (0.5+\a,0);
  \coordinate (A4) at (0.5,\b);
  \coordinate (A5) at (-0.5,\b);
  \coordinate (A6) at (-0.5-\a,0);
  \coordinate (A7) at (0,-\b-\c);
  \coordinate (A8) at (-0.5-\a-\d,-\c);
  \coordinate (A9) at (-0.5-\a-\d,\c);
  \coordinate (A10) at (0.5+\a+\d,-\c);
  \coordinate (A11) at (0.5+\a+\d,\c);
  \coordinate (A12) at (0,\b+\c);

  \draw (A1) -- (A2) -- (A3) -- (A4) --(A5)--(A6)-- cycle;
  \draw (A1) -- (A7) -- (A2) -- (A10) -- (A3) -- (A11) -- (A4) -- (A12) -- (A5) -- (A9) -- (A6)
	-- (A8)  -- cycle;

  \draw[very thin] (-0.06-\a,0) arc (0:75:.35);

  \node[right] at (-0.6,0.4)  {\small $\frac{\pi}{3}$};	
 
 \draw[dashed] (A6) --  (-0.15-\a,0);

  \node[draw,circle,inner sep=2.2pt,fill] at (A1) {};
  \node[draw,circle,inner sep=2.2pt,fill] at (A2) {};
  \node[draw,circle,inner sep=2.2pt,fill] at (A3) {};
  \node[draw,circle,inner sep=2.2pt,fill] at (A4) {};
  \node[draw,circle,inner sep=2.2pt,fill] at (A5) {};
  \node[draw,circle,inner sep=2.2pt,fill] at (A6) {};
  \node[draw,circle,inner sep=2.2pt,fill] at (A7) {};
  \node[draw,circle,inner sep=2.2pt,fill] at (A8) {};
  \node[draw,circle,inner sep=2.2pt,fill] at (A9) {};
  \node[draw,circle,inner sep=2.2pt,fill=white] at (A10) {}; 
  \node[draw,circle,inner sep=2.2pt,fill=white] at (A11) {};
  \node[draw,circle,inner sep=2.2pt,fill=white] at (A12) {};

  \end{tikzpicture}
  \end{minipage}
  &
  \begin{minipage}{.3\textwidth}
  \begin{tikzpicture}[scale=1] 
 
  \clip (-2.5,-2.2) rectangle (2.2cm,2.2cm); 
  
  \def\a{0.5};
  \def\b{0.866};
  \def\c{0.866};
  \def\d{0.5};

  \draw[line width=1.4, ->,dashed] (-0.5-\a-\d, -\b-\c) -- (0.5+\a+\d, -\b-\c);	
  \draw[line width=1.4, ->,dashed] (-0.5-\a-\d, -\b-\c) -- (-0.5-\a-\d, \b+\c);
 
  \end{tikzpicture}
  \end{minipage}
  & 
  \begin{minipage}{.3\textwidth}
  \begin{tikzpicture}[scale=0.3] 
 
  \clip (-6.8,-4.8) rectangle (9.6cm,8.5cm);    
  
  \def\a{0.5};
  \def\b{0.866};
  \def\c{0.866};
  \def\d{0.5};
  
  \foreach \x in {-9, -6, -3, 0, 3, 6, 9}
  {
   \foreach \y in {-10.392, -6.9282, -3.4641, 0, 3.4641, 6.9282, 10.392}
   {
   \coordinate (A1) at (-0.5+\x,-\b+\y);
   \coordinate (A2) at (0.5+\x,-\b+\y);
   \coordinate (A3) at (0.5+\a+\x,0+\y);
   \coordinate (A4) at (0.5+\x,\b+\y);
   \coordinate (A5) at (-0.5+\x,\b+\y);
   \coordinate (A6) at (-0.5-\a+\x,0+\y);
   \coordinate (A7) at (0+\x,-\b-\c+\y);
   \coordinate (A8) at (-0.5-\a-\d+\x,-\c+\y);
   \coordinate (A9) at (-0.5-\a-\d+\x,\c+\y);
   \coordinate (A10) at (0.5+\a+\d+\x,-\c+\y);
   \coordinate (A11) at (0.5+\a+\d+\x,\c+\y);
   \coordinate (A12) at (0+\x,\b+\c+\y);

   \draw (A1) -- (A2) -- (A3) -- (A4) --(A5)--(A6)-- cycle;
   \draw (A1) -- (A7) -- (A2) -- (A10) -- (A3) -- (A11) -- (A4) -- (A12) -- (A5) -- (A9) -- (A6)
	-- (A8)  -- cycle;
	
   \node[draw,circle,inner sep=1pt,fill] at (A1) {};
   \node[draw,circle,inner sep=1pt,fill] at (A2) {};
   \node[draw,circle,inner sep=1pt,fill] at (A3) {};
   \node[draw,circle,inner sep=1pt,fill] at (A4) {};
   \node[draw,circle,inner sep=1pt,fill] at (A5) {};
   \node[draw,circle,inner sep=1pt,fill] at (A6) {};
   \node[draw,circle,inner sep=1pt,fill] at (A7) {};
   \node[draw,circle,inner sep=1pt,fill] at (A8) {};
   \node[draw,circle,inner sep=1pt,fill] at (A9) {};
   \node[draw,circle,inner sep=1pt,fill] at (A10) {};
   \node[draw,circle,inner sep=1pt,fill] at (A11) {};
   \node[draw,circle,inner sep=1pt,fill] at (A12) {};
   }
  }
  \end{tikzpicture}
  \end{minipage} 
  \\ \hline
  \end{tabular}
  \caption{A crystal framework with linearly decomposable RUM spectrum.}
  \label{CFHex}
\end{table}

\begin{Example}
\label{NonRegHex}
Let $\C$ be the crystal framework illustrated in Table \ref{CFNonRegHex}.
This framework is in Maxwell counting equilibrium and the  symbol function $\Phi_\C(z,w)$ is, 
\[{\small 
\left[ \begin {array}{cccccccccccccccccc} -1&0&1&0&0&0&0&0&0&0&0&0&0&0
&0&0&0&0\\ \noalign{\medskip}0&0&-b&-a&b&a&0&0&0&0&0&0&0&0&0&0&0&0\\ \noalign{\medskip}0&0&0&0&b&-a&-b&a&0&0&0&0&0&0&0&0&0&0\\ \noalign{\medskip}0&0&0&0&0
&0&1&0&-1&0&0&0&0&0&0&0&0&0\\ \noalign{\medskip}0&0&0&0&0&0&0&0&b&a&-b&-a&0&0&0&0&0&0\\ \noalign{\medskip}
b&-a&0&0&0&0&0&0&0&0&-b&a&0&0&0&0&0&0
\\ \noalign{\medskip}-\frac{1}{2}&\frac{\sqrt {3}}{2}&0&0&0&0&0&0&0&0&0&0&\frac{1}{2}&-\frac{\sqrt {3}}{2}&0&0&0&0\\ \noalign{\medskip}0&0&\frac{1}{2}&\frac{\sqrt {3}}{2}&0&0&0&0
&0&0&0&0&-\frac{1}{2}&-\frac{\sqrt {3}}{2}&0&0&0&0\\ \noalign{\medskip}0&0&-a&b&0&0&0&0&0&0&0&0&0&0&a\bar{z}&
-b\bar{z}&0&0\\ \noalign{\medskip}0&0&0&0&-\frac{1}{\sqrt {2}}&\frac{1}{\sqrt {2}}&0&0&0
&0&0&0&0&0&\frac{\bar{z}}{\sqrt{2}}&-\frac{\bar{z}}{\sqrt{2}}&0&0
\\ \noalign{\medskip}0&0&0&0&-\frac{1}{\sqrt {2}}&-\frac{1}{\sqrt {2}}&0&0&0&0&0
&0&0&0&0&0&\frac{\bar{z}}{\sqrt{2}}&\frac{\bar{z}}{\sqrt{2}}
\\ \noalign{\medskip}0&0&0&0&0&0&-a&b&0&0&0&0&0&0&0&0&a\bar{z}&-b\bar{z}\\ \noalign{\medskip}0&0&0&0&0&0
&\frac{1}{2}&-\frac{\sqrt {3}}{2}&0&0&0&0&-\frac{\bar{w}}{2}&\frac {\sqrt {3}\bar{w}}{2}&0&0&0&0\\ \noalign{\medskip}0&0&0&0&0&0&0&0&-\frac{1}{2}&-\frac{\sqrt {3}}{2}&0&0
&\frac{\bar{w}}{2}&\frac {\sqrt {3}\bar{w}}{2}&0&0&0&0
\\ \noalign{\medskip}0&0&0&0&0&0&0&0&a&b&0&0&0&0&0&0&-a&-b\\ \noalign{\medskip}0&0&0&0&0&0&0&0&0&0&\frac{1}{\sqrt {2}}&-\frac{1}{\sqrt {2}}&0&0&0&0&-\frac{1}{\sqrt {2}}&\frac{1}{\sqrt {2}}\\ \noalign{\medskip}0&0&0&0
&0&0&0&0&0&0&\frac{1}{\sqrt {2}}&\frac{1}{\sqrt {2}}&0&0&-\frac{1}{\sqrt {2}}&-\frac{1}{\sqrt {2}}&0&0\\ \noalign{\medskip}a&-b&0&0&0&0&0&0&0&0&0&0&0&0&-a&b&0&0\end {array} \right] 
}\]
where $a=\frac{\sqrt{3}+1}{2\sqrt{2}}$ and $b=\frac{\sqrt{3}-1}{2\sqrt{2}}$.
The crystal polynomial takes the form,
\begin{align*}
p_{\mathcal{C}}(z,w) 
&=
z^4w-\frac{1}{\sqrt{3}}z^3w^2+\left(\frac{\sqrt{3}}{2}-2\right)z^3w-\frac{1}{2\sqrt{3}}z^3 +\left(\frac{1}{2}+\frac{1}{2\sqrt{3}}\right)z^2w^2 \\ 
&+\frac{1}{2\sqrt{3}}z^2w-\frac{1}{\sqrt{3}}z^2 -\frac{1}{2}zw^2+\left(\frac{3}{2}-\frac{1}{\sqrt{3}}\right)zw+\frac{1}{\sqrt{3}}z
-\frac{1}{2}w
\end{align*}
which leads to a finite RUM spectrum.
\end{Example}
\vspace{5mm}

\begin{table}[h]
   \centering
  \begin{tabular}{ | c | c | c | }
  \hline
  Motif & Translation group & Framework \\ 
  \hline
  \begin{minipage}{.3\textwidth}
  \begin{tikzpicture}[scale=1]
 
  \clip (-2.6,-2.2) rectangle (2cm,2.2cm); 
  
  \def\a{0.2588};
  \def\b{0.9659};
  \def\c{0.866};
  \def\d{0.7071};

  \coordinate (A1) at (-0.5,-\b);
  \coordinate (A2) at (0.5,-\b);
  \coordinate (A3) at (0.5+\a,0);
  \coordinate (A4) at (0.5,\b);
  \coordinate (A5) at (-0.5,\b);
  \coordinate (A6) at (-0.5-\a,0);
  \coordinate (A7) at (0,-\b-\c);
  \coordinate (A8) at (-0.5-\a-\d,-\d);
  \coordinate (A9) at (-0.5-\a-\d,\d);
  \coordinate (A10) at (0.5+\a+\d,-\d);
  \coordinate (A11) at (0.5+\a+\d,\d);
  \coordinate (A12) at (0,\b+\c);

  \draw (A1) -- (A2) -- (A3) -- (A4) --(A5)--(A6)-- cycle;
  \draw (A1) -- (A7) -- (A2) -- (A10) -- (A3) -- (A11) -- (A4) -- (A12) -- (A5) -- (A9) -- (A6)
	-- (A8)  -- cycle;
  \draw[very thin] (-0.15-\a,0) arc (0:75:.35);
 
  \node[right] at (-0.6,0.4)  {\small $\frac{5\pi}{12}$};	 
  \draw[dashed] (A6) --  (-0.15-\a,0);

  \node[draw,circle,inner sep=2.2pt,fill] at (A1) {};
  \node[draw,circle,inner sep=2.2pt,fill] at (A2) {};
  \node[draw,circle,inner sep=2.2pt,fill] at (A3) {};
  \node[draw,circle,inner sep=2.2pt,fill] at (A4) {};
  \node[draw,circle,inner sep=2.2pt,fill] at (A5) {};
  \node[draw,circle,inner sep=2.2pt,fill] at (A6) {};
  \node[draw,circle,inner sep=2.2pt,fill] at (A7) {};
  \node[draw,circle,inner sep=2.2pt,fill] at (A8) {};
  \node[draw,circle,inner sep=2.2pt,fill] at (A9) {};
  \node[draw,circle,inner sep=2.2pt,fill=white] at (A10) {}; 
  \node[draw,circle,inner sep=2.2pt,fill=white] at (A11) {};
  \node[draw,circle,inner sep=2.2pt,fill=white] at (A12) {};

  \end{tikzpicture}
  \end{minipage}
  &
  \begin{minipage}{.3\textwidth}
  \begin{tikzpicture}[scale=1] 
  \clip (-2.6,-2.2) rectangle (2.2cm,2.2cm); 
  
  \def\a{0.2588};
  \def\b{0.9659};
  \def\c{0.866};
  \def\d{0.7071};

  \draw[line width=1.4,->,dashed] (-0.5-\a-\d, -\b-\c) -- (0.5+\a+\d, -\b-\c);	
  \draw[line width=1.4,->,dashed] (-0.5-\a-\d, -\b-\c) -- (-0.5-\a-\d, \b+\c);
  \end{tikzpicture}
  \end{minipage}
  & 
  \begin{minipage}{.3\textwidth}
  \begin{tikzpicture}[scale=0.3] 
  \clip (-7.8,-4.8) rectangle (9cm,8.5cm);    
  
  \def\a{0.2588};
  \def\b{0.9659};
  \def\c{0.866};
  \def\d{0.7071};

  \foreach \x in {-8.796, -5.8637, -2.9319, 0, 2.9319, 5.8637, 8.796}
  {
   \foreach \y in {-10.9917, -7.3278, -3.6639, 0, 3.6639, 7.3278, 10.9917}
   {
   \coordinate (A1) at (-0.5+\x,-\b+\y);
   \coordinate (A2) at (0.5+\x,-\b+\y);
   \coordinate (A3) at (0.5+\a+\x,0+\y);
   \coordinate (A4) at (0.5+\x,\b+\y);
   \coordinate (A5) at (-0.5+\x,\b+\y);
   \coordinate (A6) at (-0.5-\a+\x,0+\y);
   \coordinate (A7) at (0+\x,-\b-\c+\y);
   \coordinate (A8) at (-0.5-\a-\d+\x,-\d+\y);
   \coordinate (A9) at (-0.5-\a-\d+\x,\d+\y);
   \coordinate (A10) at (0.5+\a+\d+\x,-\d+\y);
   \coordinate (A11) at (0.5+\a+\d+\x,\d+\y);
   \coordinate (A12) at (0+\x,\b+\c+\y);
 
   \draw (A1) -- (A2) -- (A3) -- (A4) --(A5)--(A6)-- cycle;
   \draw (A1) -- (A7) -- (A2) -- (A10) -- (A3) -- (A11) -- (A4) -- (A12) -- (A5) -- (A9) -- (A6)
	-- (A8)  -- cycle;
	
   \node[draw,circle,inner sep=1pt,fill] at (A1) {};
   \node[draw,circle,inner sep=1pt,fill] at (A2) {};
   \node[draw,circle,inner sep=1pt,fill] at (A3) {};
   \node[draw,circle,inner sep=1pt,fill] at (A4) {};
   \node[draw,circle,inner sep=1pt,fill] at (A5) {};
   \node[draw,circle,inner sep=1pt,fill] at (A6) {};
   \node[draw,circle,inner sep=1pt,fill] at (A7) {};
   \node[draw,circle,inner sep=1pt,fill] at (A8) {};
   \node[draw,circle,inner sep=1pt,fill] at (A9) {};
   \node[draw,circle,inner sep=1pt,fill] at (A10) {};
   \node[draw,circle,inner sep=1pt,fill] at (A11) {};
   \node[draw,circle,inner sep=1pt,fill] at (A12) {};
   }
  }

  \end{tikzpicture}
  \end{minipage} 
  \\ \hline
  \end{tabular}
  \caption{A crystal framework with  finite RUM spectrum}
\label{CFNonRegHex}
\end{table}

\begin{Example}
\label{OctagonEx}
The crystal framework illustrated in Table \ref{Octagon}  is based on a motif consisting of a regular octagon of equilteral triangles. The crystal polynomial is, 
\[p_\C(z,w) =p_1(z,w)p_2(z,w)\]
where
\[p_1(z,w)=(\sqrt{3}-\sqrt{2})  z^2w-zw^2 +2( \sqrt{2}-\sqrt{3} +1) zw+(\sqrt{3}-\sqrt{2}) w-z\]
\[p_2(z,w)= (\sqrt{3} +\sqrt{2}) z^2w-zw^2 -2 (\sqrt{2}+ \sqrt{3} -1) zw+(\sqrt{3} +\sqrt{2}) w-z\]
The RUM spectrum consists of points $(z,w)\in\bT^2$ which satisfy $\Re(w)=a\Re(z)+(1-a)$ for either $a=\sqrt{3}-\sqrt{2}$ or $a=\sqrt{3}+\sqrt{2}$.
This set is illustrated in Table \ref{OctagonRUM} as a subset of the torus $[0,1)^2$ which consists of four closed curves with the common intersection point $(0,0)$.
\end{Example}

\begin{table}[h]
   \centering
  \begin{tabular}{ | c | c | c | }
  \hline
  Motif & Translation group & Framework \\ 
  \hline
  \begin{minipage}{.3\textwidth}
    \begin{tikzpicture}[scale=1]
 
  \clip (-2.6,-2.2) rectangle (2.2cm,2.2cm); 
  
  \def\a{0.382};
  \def\b{0.9239};
  \def\c{0.6088};
  \def\d{0.7934};
  \def\e{1.3066};

  \coordinate (A1) at (-\b,-\b);
  \coordinate (A2) at (0,-\e);
  \coordinate (A3) at (\b,-\b);
  \coordinate (A4) at (\e,0);
  \coordinate (A5) at (\b,\b);
  \coordinate (A6) at (0,\e);
  \coordinate (A7) at (-\b,\b);
  \coordinate (A8) at (-\e,0);
  \coordinate (A9) at (-\e-\c,-\d);
  \coordinate (A10) at (\d,-\e-\c);
  \coordinate (A11) at (\e+\c,\d);
  \coordinate (A12) at (-\d,\e+\c);
  \coordinate (A13) at (-\e-\c,\d);
  \coordinate (A14) at (-\d,-\e-\c);
  \coordinate (A15) at (\e+\c,-\d);
  \coordinate (A16) at (\d,\e+\c);

  \draw (A1) -- (A2) -- (A3) -- (A4) --(A5)--(A6)--(A7)--(A8)-- cycle;
  \draw (A1) -- (A14) -- (A2) -- (A10) -- (A3) -- (A15) -- (A4) -- (A11) -- (A5) -- (A16) -- (A6)
	-- (A12) -- (A7) -- (A13) -- (A8) -- (A9) -- cycle;

  \node[draw,circle,inner sep=2.2pt,fill] at (A1) {};
  \node[draw,circle,inner sep=2.2pt,fill] at (A2) {};
  \node[draw,circle,inner sep=2.2pt,fill] at (A3) {};
  \node[draw,circle,inner sep=2.2pt,fill] at (A4) {};
  \node[draw,circle,inner sep=2.2pt,fill] at (A5) {};
  \node[draw,circle,inner sep=2.2pt,fill] at (A6) {};
  \node[draw,circle,inner sep=2.2pt,fill] at (A7) {};
  \node[draw,circle,inner sep=2.2pt,fill] at (A8) {};
  \node[draw,circle,inner sep=2.2pt,fill] at (A9) {};
  \node[draw,circle,inner sep=2.2pt,fill] at (A10) {}; 
  \node[draw,circle,inner sep=2.2pt,fill] at (A11) {};
  \node[draw,circle,inner sep=2.2pt,fill] at (A12) {};
  \node[draw,circle,inner sep=2.2pt,fill=white] at (A13) {};
  \node[draw,circle,inner sep=2.2pt,fill=white] at (A14) {};
  \node[draw,circle,inner sep=2.2pt,fill=white] at (A15) {};
  \node[draw,circle,inner sep=2.2pt,fill=white] at (A16) {};  

  \end{tikzpicture}
  \end{minipage}
  &
  \begin{minipage}{.3\textwidth}
 \begin{tikzpicture}[scale=1] 
 
   \clip (-2.6,-2.2) rectangle (2.2cm,2.2cm); 
  
   \def\c{0.6088}; 
   \def\e{1.3066};
 
   \draw[line width =1.2,->,dashed] (-\e-\c, -\e-\c) --  (\e+\c, -\e-\c);	
   \draw[line width =1.2,->,dashed] (-\e-\c, -\e-\c) -- (-\e-\c, \e+\c);
  \end{tikzpicture}
  \end{minipage}
  & 
  \begin{minipage}{.3\textwidth}
  \begin{tikzpicture}[scale=0.3] 
 
    \clip (-4.8,-4.8) rectangle (12cm,8.5cm);    
  
  \foreach \x in {-11.49, -7.66, -3.83, 0, 3.83, 7.66, 11.49}
  {
	  
  \foreach \y in {-11.49, -7.66, -3.83, 0, 3.83, 7.66, 11.49}
  {
   
  \def\a{0.382};
  \def\b{0.9239};
  \def\c{0.6088};
  \def\d{0.7934};
  \def\e{1.3066};

  \coordinate (A1) at (-\b+\x,-\b+\y);
  \coordinate (A2) at (0+\x,-\e+\y);
  \coordinate (A3) at (\b+\x,-\b+\y);
  \coordinate (A4) at (\e+\x,0+\y);
  \coordinate (A5) at (\b+\x,\b+\y);
  \coordinate (A6) at (0+\x,\e+\y);
  \coordinate (A7) at (-\b+\x,\b+\y);
  \coordinate (A8) at (-\e+\x,0+\y);
  \coordinate (A9) at (-\e-\c+\x,-\d+\y);
  \coordinate (A10) at (\d+\x,-\e-\c+\y);
  \coordinate (A11) at (\e+\c+\x,\d+\y);
  \coordinate (A12) at (-\d+\x,\e+\c+\y);
  \coordinate (A13) at (-\e-\c+\x,\d+\y);
  \coordinate (A14) at (-\d+\x,-\e-\c+\y);
  \coordinate (A15) at (\e+\c+\x,-\d+\y);
  \coordinate (A16) at (\d+\x,\e+\c+\y);

  \draw (A1) -- (A2) -- (A3) -- (A4) --(A5)--(A6)--(A7)--(A8)-- cycle;
  \draw (A1) -- (A14) -- (A2) -- (A10) -- (A3) -- (A15) -- (A4) -- (A11) -- (A5) -- (A16) -- (A6)
	-- (A12) -- (A7) -- (A13) -- (A8) -- (A9) -- cycle;
 
  \node[draw,circle,inner sep=1pt,fill] at (A1) {};
  \node[draw,circle,inner sep=1pt,fill] at (A2) {};
  \node[draw,circle,inner sep=1pt,fill] at (A3) {};
  \node[draw,circle,inner sep=1pt,fill] at (A4) {};
  \node[draw,circle,inner sep=1pt,fill] at (A5) {};
  \node[draw,circle,inner sep=1pt,fill] at (A6) {};
  \node[draw,circle,inner sep=1pt,fill] at (A7) {};
  \node[draw,circle,inner sep=1pt,fill] at (A8) {};
  \node[draw,circle,inner sep=1pt,fill] at (A9) {};
  \node[draw,circle,inner sep=1pt,fill] at (A10) {};
  \node[draw,circle,inner sep=1pt,fill] at (A11) {};
  \node[draw,circle,inner sep=1pt,fill] at (A12) {};
    }
 }

  \end{tikzpicture}
  \end{minipage} 
  \\ \hline
  \end{tabular}
  \caption{A crystal framework with linearly indecomposable RUM spectrum.}
\label{Octagon}
\end{table}

 \begin{table}[h]
 \centering
 \includegraphics[scale=0.5]{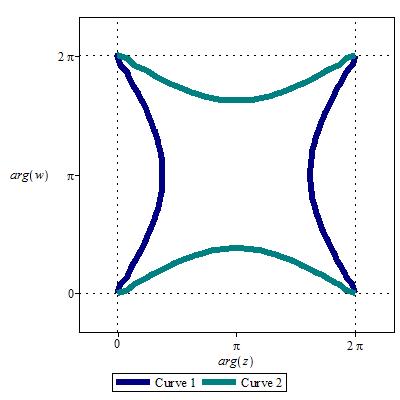} 
 \caption{The RUM spectrum for Example \ref{OctagonEx}.}
 \label{OctagonRUM}
 \end{table}

Let $\tilde{\C}$ be the basic one-dimensional grid framework for the lattice $\bZ$ in $\bR$. For any crystal framework $\C$ in $\bR^d$ one may construct a \emph{product framework}
$\tilde{\C}= \C \times \C_{\bZ}$  in $\bR^{d+1}$  whose intersection with 
the hyperplanes $\bR^d \times \{n\}$ are copies of $\C$ and where these copies are connected
by the edges  $((p(v),n),(p(v),n+1))$. In the case that  $\C$ has square matrix symbol function $\Phi_\C(z_1,\dots ,z_d)$ it is straightforward to verify that
\[
p_{\tilde{\C}}(z_1,\dots ,z_{d+1})=(z_{d+1}-1)^{|F_v|}p_\C(z_1,\dots ,z_d)
\]
This leads readily to the identification of the RUM spectrum in $\bT^3$ of such frameworks. Further three-dimensional examples not of this product form may be found in Power \cite{pow-poly} and Wegner \cite{weg}.

\begin{Example}
Let $\tilde{\C}$ be the three dimensional framework derived from the regular octagon framework of Example \ref{Octagon}. Then the crystal polynomial admits a three-fold factorisation and from this it follows that the RUM spectrum has the topological structure of four two-dimensional tori connected over the common circle of points $(1,1,z)$ in $\bT^3$.

\end{Example}

%%%%%%%%%%%%%%%%%%%%%%%%%%%%%%%%%%%%%%%%%%%%%%%

%%%%%%%%%%%%%%%%%%%%%%%%%%%%%%%%%%%%%%%%%%%%%%%

\end{document}